\documentclass[12pt]{amsart}
\usepackage{amsfonts}
\usepackage{mathrsfs}
\usepackage{amsmath}
\usepackage{amsmath, amsthm, amssymb}
\usepackage{color}
\numberwithin{equation}{section}
\newcommand{\reff}[1]{(\ref{#1})}

\newcommand{\T}{\theta}
\newcommand{\VT}{\vartheta}
\renewcommand{\(}{\left( }
\renewcommand{\)}{\right) }

\renewcommand{\theequation}{\theequation. \arabic{equation}}
\numberwithin{equation}{section}
\newtheorem{thm}{Theorem}[section]
\newtheorem{cor}[thm]{Corollary}

\newtheorem{prop}[thm]{Proposition}
\newtheorem{defn}[thm]{Definition}
\newtheorem{conj}[thm]{Conjecture}

\def\squarebox#1{\hbox to #1{\hfill\vbox to #1{\vfill}}}

\begin{document}
\title[A theta function identity of degree eight]
{ A universal identity for theta functions of degree eight  and applications }
\author{Zhi-Guo Liu }
\address{School of Mathematical Sciences and  Shanghai  Key Laboratory of PMMP, East China Normal University,
	500 Dongchuan Road, Shanghai 200241, P.R.
	China}\email{zgliu@math.ecnu.edu.cn, liuzg@hotmail.com}
\thanks{This work was supported by the National Science Foundation of China (Grant No. 11971173) and Science and Technology Commission of Shanghai Municipality (Grant No. 13dz2260400).}
\begin{abstract}
	Previously, we proved an identity for theta functions of degree eight,  and several applications of it were also discussed. This identity is a natural extension of the addition formula for the Weierstrass sigma-function.
	In this paper we will use this identity to reexamine our work in theta function identities in the past two decades. Hundreds of results about elliptic modular functions, both classical and new, are derived from this identity with ease. Essentially, this general theta function identity is a theta identities generating machine. Our investigation shows that many well-known results about elliptic modular functions  with different appearances due to Jacobi, Kiepert, Ramanujan and Weierstrass among others,  actually  share a  common source.  This paper can also be seen as a summary of  my past work on theta function identities. A conjecture is also proposed.
	
\end{abstract}
\dedicatory{Dedicated to Srinivasa Ramanujan on the occasion of his 133rd birth anniversary }
\maketitle	
%%%%%%%%%%%%%%%%%%%%%%%%%%%%%%%%%%%%%%%%%%%%%%%%%%%%%%%%%%%%%%%%%%%%%%%%%
%%%%%%%%%%%%%%%%%%%%%%%%%%%%%%%%%%%%%%%%%%%%%%%%%%%%%%%%%%%%%%%%%%%%%%%%%%%%%%%%
\section{Introduction and preliminary}
%%%%%%%%%%%%%%%%%%%%%%%%%%%%%%%%%%%%%%%%%%%%%%%%%%%%%%%%%%%%%%%%%%%%%%%%%%%%%%%%%%%%%%%%%%%%%%%%%%%%%%%%%%%%%%%%%%%%%%
 For convenience, sometimes we use $\exp(z)$ to denote the natural exponential function $e^{z}$.  Throughout this paper we take $q=\exp(2\pi i \tau)$, where $i$ is the imaginary unit and $\tau$ has positive imaginary part. So that we have $|q|<1$.  

The Dedekind eta function is a modular form of weight $1/2$ which is  defined by 
\begin{equation}\label{jabel:eqn1}
	\eta(\tau)=q^{1/24} \prod_{n=1}^\infty (1-q^n)=e^{\frac{\pi i\tau}{12}}\prod_{n=1}^\infty (1-e^{2\pi n i\tau}).
\end{equation} 

To carry out our study, we need the Jacobi theta function $\T_1(z|\tau)$ which is defined as (see, for example  \cite[p.~463]{WhiWat})
\begin{align}\label{jabel:eqn2}
	\T_1(z|\tau)&=-iq^{1/8}\sum_{n=-\infty}^\infty (-1)^n q^{n(n+1)/2} e^{(2n+1)iz}\\
	&=2q^{1/8}\sum_{n=0}^\infty (-1)^n 	q^{n(n+1)/2} \sin (2n+1) z.\nonumber
\end{align}

Jacobi's  triple product identity in the next theorem is one of the most fundamental  results in the theory of elliptic theta functions and $q$-series, which can be found in any  standard textbook in elliptic theta functions or $q$-series  (see, for example \cite[Theorem~1.3.3]{Berndt06} and \cite[Theorem~10.4.1]{AAR99}).
\begin{thm}[Jacobi triple product identity]\label{jactripthm}  For $z\not=0$ and $|q|<1$, we have 
	\begin{equation}\label{jabel:eqn3}
		(1-z)\prod_{n=1}^\infty (1-q^n)(1-q^n z) (1-q^n/z)=\sum_{n=-\infty}^\infty (-1)^n q^{n(n-1)/2} z^n.
	\end{equation}	
\end{thm}

Using the Jacobi triple product identity, one can get the infinite product representation for $\T_1$,
\begin{equation}\label{jabel:eqn4}
	\T_1(z|\tau)=2q^{1/8}(\sin z )\prod_{n=1}^\infty (1-q^n)(1-q^ne^{2iz}) (1-q^ne^{-2iz}).
\end{equation}

It is easily seen that in the fundamental periodic parallelogram
given by
\begin{equation}\label{jabel:eqn5}
	\prod=\{x\pi+y\pi\tau|~ 0\le x< 1, 0\le y< 1\},
\end{equation}
the zeros of $\theta_1(z|\tau)$ are at
$z=0$.  The set of zeros of $\theta_1(z|\tau)$  form a lattice $\Lambda$, which is given by 
\begin{equation}\label{jabel:eqn6}
	\Lambda=\{ m \pi+ n\pi \tau: (m, n) \in \mathbb{Z}^2 \}.
\end{equation}

The Jacobi theta function $\theta_1(z|\tau)$ is an entire, quasi-doubly periodic function of $z$,  and  regarded as the two dimensional version of the Sine function. 

\begin{defn} \label{qequiv}
	If the difference of two complex numbers is in the set $\Lambda$,  then these two complex numbers are 
	said to be equivalent modulo  $ \Lambda$, and if the difference of two complex numbers is not in the set $\Lambda$,  then these two complex numbers are said to be inequivalent modulo  $ \Lambda$.
\end{defn}

The Bernoulli numbers $B_k$ are defined as the coefficients in the power series 
\begin{equation}\label{jabel:eqn7}
	\frac{z}{e^{z}-1}=\sum_{k=0}^\infty B_k \frac{z^k}{k!}~\quad \text{for}\quad 
	|z|<2\pi,
\end{equation}
and the normalized  Eisenstein series $E_{2k}(\tau)$ on the full modular group are defined by  \cite[Eq.~(6.1.4)]{Rankin1977}
\begin{equation}\label{jabel:eqn8}
	E_{2k}(\tau)=1-\frac{4k}{B_{2k}}\sum_{n=1}^\infty \frac{n^{2k-1}q^n}{1-q^n}.	
\end{equation}
	
For simplicity,  we will use $L(\tau), ~M(\tau)$ and $N(\tau)$  to denote $E_2(\tau), ~E_4(\tau)$ and $E_6(\tau)$ respectively. Thus we have
 \cite[p.~195]{Rankin1977}
\begin{equation}\label{jabel:eqn9}
	\begin{split}
	L(\tau):=	E_2(\tau)&=1-24\sum_{n=1}^\infty \frac{nq^n}{1-q^n},\\
	M(\tau):=	E_4(\tau)&=1+240\sum_{n=1}^\infty \frac{n^3 q^n}{1-q^n},\\
	N(\tau):=	E_6(\tau)&=1-504\sum_{n=1}^\infty \frac{n^5 q^n}{1-q^n}.
	\end{split}	
\end{equation}

It is known that the Weierstrass elliptic function $\wp(z|\tau)$ attached to the periodic lattice $\Lambda$ is defined by  \cite[p.~10]{Apostol1990}
\begin{equation}\label{jabel:eqn10}
	\wp(z|\tau)=\frac{1}{z^2}+\sum_{{\omega \in \Lambda}\atop{\omega \not=0}}
	\( \frac{1}{(z-\omega)^2}-\frac{1}{\omega^2}\),
\end{equation}
which has primitive periods $\pi$ and $\pi\tau$. Also it has only one inequivalent pole at $z=0$, of order two.

For $k>2$, the Eisenstein series $G_k(\tau)$ attached to the Lattice $\Lambda$ is defined by
\begin{equation}\label{jabel:eqn11}
G_k(\tau)=\sum_{(m, n)\not=(0, 0)} \frac{1}{(m+n \tau)^k}.
\end{equation}
It is well-known that the Laurent expansion of $\wp(z|\tau)$ near the origin is given by \cite[Theorem~1.11]{Apostol1990}
\begin{equation}\label{jabel:eqn12}
	\wp(z|\tau)=\frac{1}{z^2}+\sum_{k=1}^\infty (2k+1)G_{2k+2}(\tau) z^{2k+2}.
\end{equation}

For any function $f(z|\tau),$  we will use the prime and the double prime  to denote the  first order  and the second order partial derivatives of $f(z|\tau)$ with respect $z$, etc.   We sometimes use $\(\log f\)'(z|\tau)$  and $\(\log f\)''(z|\tau)$to denote the first order and the second order partial logarithmic derivatives of $f(z|\tau)$ with respect to $z$, etc. 

Logarithmically differentiating  the infinite product representations of $\T_1(z|\tau)$ with respect to $z$, respectively, one can  find that (see, for example  \cite[p.~45]{Bellman})
\begin{equation} \label{jabel:eqn13}
	(\log \T_1)'(z|\tau)=\cot z+4\sum_{n=1}^\infty \frac{q^n}{1-q^n} \sin 2nz.	
\end{equation}

Substituting the Laurent series expansion of $\cot z$ near $z=0$ and the Maclaurin series of $\sin z$ to the right-hand side of the above equation, we easily find that near $z=0,$
\begin{equation}\label{jabel:eqn14}
	(\log \T_1)'(z|\tau)=\frac{1}{z}-\frac{1}{3}L(\tau)z-\frac{1}{45}M(\tau)z^3-
	\frac{2}{945}N(\tau)z^5+O(z^7),
\end{equation}
 where $L(\tau), M(\tau)$ and $N(\tau)$ are defined by \reff{jabel:eqn9}.

It is not difficult to verify that the Weierstrass elliptic function 
is related to the second order partial logarithmic derivative of $\T_1$  by the relation
\begin{align}\label{jabel:eqn15}
	\wp(z|\tau)=-\(\log \T_1\)''(z|\tau)-\frac{1}{3}L(\tau).
\end{align}		
In this paper we also need the Jacobi theta function $\T_2,~\T_3$ and $\T_4$ which are defined as follows:		
\begin{defn} \label{jtheta}The Jacobi theta
	functions $\T_k$ for $k=2, 3, 4,$ are defined as
	\begin{align*}
		\T_2(z|\tau)&=2\sum_{n=0}^\infty q^{\frac{(2n+1)^2}{8}} ~  \cos (2n+1)z,\\
		\T_3(z | \tau)&=1+2\sum_{n=1}^\infty q^{\frac{1}{2}n^2} \cos 2nz ,\\
		\T_4(z | \tau)&=1+2\sum_{n=1}^\infty (-1)^n q^{\frac{1}{2}n^2}  \cos 2nz.
	\end{align*}
\end{defn}

Using the Jacobi triple product identity one can easily derive the infinite  product representations of the Jacobi theta functions in the following proposition.
\begin{prop}\label{infiniteprod}  Let $ \T_2, \T_3$ and $\T_4$ be defined by Definition~\ref{jtheta}. Then  
	\begin{align*}
		\T_2(z|\tau)&=2q^{1/8}(\cos z )\prod_{n=1}^\infty (1-q^n)(1+q^ne^{2iz}) (1+q^ne^{-2iz}),\\
		\T_3(z|\tau)&=\prod_{n=1}^\infty (1-q^n)(1+q^{(n-1/2)} e^{2iz}) (1+q^{(n-1/2)}e^{-2iz}),\\
		\T_4(z|\tau)&=\prod_{n=1}^\infty (1-q^n)(1-q^{(n-1/2)} e^{2iz}) (1-q^{(n-1/2)}e^{-2iz}).	
	\end{align*}
\end{prop}
Using the infinite product expansions of the Jacobi theta functions, we can easily find the following multiplication formulas for theta functions. These formulas were proved by Jacobi \cite{Jacobi1828a} in 1828 (see, also \cite[pp.~300--320]{Enne 1890}).
\begin{prop}\label{multhetapp}
	If $n$ is an odd integer, then for $j=1, 2, 3, 4,$ we have
	\begin{equation}\label{jabel:eqn16}
		\T_j(z|\tau) \prod_{k=1}^{\frac{n-1}{2}} \T_j \(\frac{k\pi }{n}+z|\tau \) \T_j \(\frac{k\pi }{n}-z|\tau \)=\frac{\eta^n(\tau)}{\eta(n\tau)} 
		{\T_j(nz|n\tau)},
	\end{equation}
	and
	\begin{equation}\label{jabel:eqn17}
		\T_j(z|\tau)\prod_{k=1}^{\frac{n-1}{2}} \T_j \(z+\frac{k\pi \tau }{n}|\tau \) \T_j \(z-\frac{k\pi \tau}{n}|\tau \)=q^{\frac{(1-n^2)}{24n}}\frac{\eta^n (\tau)}{\eta(\frac{\tau}{n})}
		\T_j \(z|\frac{\tau}{n}\).
	\end{equation}
\end{prop}
The four Jacobi theta functions are mutually related, and starting from one of them we may obtain the other three by simple calculations. For example, we have the following proposition.
\begin{prop}\label{halfperiods} Theta functions $\T_1, \T_2, \T_3$ and $\T_4$ satisfy the relations
	\begin{align*}
		\T_1\(z+{\pi}/{2}\Big|\tau\)&=\T_2(z|\tau),\\
		\T_1\(z+(\pi\tau)/2\Big|\tau\)&=i q^{-\frac{1}{8}}e^{-iz} \T_4(z|\tau), \\
		\T_1\(z+(\pi+\pi\tau)/2 \Big|\tau\)&=q^{-\frac{1}{8}} e^{-iz} \T_3(z|\tau).
	\end{align*}
\end{prop}

Theta functions $\T_1, \T_2, \T_3$ and $\T_4$  are not  elliptic functions, and they satisfy the following functional equations.
\begin{prop}\label{doubleperiods}
	With respect to the
	(quasi) periods $\pi$ and $\pi\tau$, we have 
	\begin{align*}
		-\theta_1(z | \tau)&=\theta_1(z+\pi | \tau)=\exp ((2z+\pi \tau)i)
		~\theta_1(z+\pi\tau | \tau),\\
		\theta_2(z | \tau)&=-\theta_2(z+\pi | \tau)=\exp ((2z+\pi \tau)i)~
		\theta_2(z+\pi\tau | \tau), \\
		\theta_3(z | \tau)&=\theta_3(z+\pi | \tau)=\exp ((2z+\pi \tau)i)~
		\theta_3(z+\pi\tau | \tau),\\
		\theta_4(z | \tau)&=\theta_4(z+\pi | \tau)=-\exp ((2z+\pi \tau)i)~
		\theta_4(z+\pi\tau | \tau).
	\end{align*}
\end{prop}
We now introduce the concept of the degree of a theta function.
\begin{defn}\label{ddefn} Suppose that $r$ is a non-negative  integer and $a, b$
	are two  nonzero complex numbers,  and let
	$f(z)$ be an entire function of $z$   satisfying the functional equations
	$f(z+\pi)=af(z)$ and $f(z+\pi\tau)
	=be^{-2irz}f(z).$
	Then  we say $f(z)$ is a theta function of degree $r.$
\end{defn}
It is obvious that the four Jacobi theta functions are all theta functions of degree $1$. For any non-negative integer $r$, the $r$ th powers of Jacobi theta functions have degree $r$. 

For convenience, we will use  $\VT_1'(\tau)$ and $\VT_j(\tau)$ to denote  $\T_1'(0|\tau) $ and $\VT_j(0|\tau)$ for $j=2, 3, 4$, respectively.

Differentiating both sides of the infinite product representation of $\theta_1$ in Proposition~\ref{infiniteprod},  one can find that
\begin{equation}\label{jabel:eqn18}
	\VT_1'(\tau)=2q^{1/8} \prod_{n=1}^{\infty} (1-q^n)^3=2\eta^3(\tau).
\end{equation}

The four Jacobi imaginary transformation formulas of theta functions were first obtained by Jacobi in 1828, who obtained them from the theory of elliptic functions \cite[pp.~ 403-404]{Jacobi1828} (see also \cite[p.~177]{Rademacher1973} and \cite[p.~ 475]{WhiWat}),  but Poisson \cite{Poisson1827} had previously obtained a formula equivalent to the imaginary transformation formula of $\T_3$ in 1827 by using the Poisson summation formula (see also \cite[pp.~7--11]{Bellman}). The imaginary transformation formulas are among the deepest results of the elliptic theta function theory, and the imaginary transformation formulas of theta functions  are a bridge between elliptic functions and modular forms.
\begin{prop}\label{imaginarypp}  If $\operatorname{Im}(\tau)>0$ and $\sqrt{-\tau i}=+1$ for $\tau=i$, then we have
	\begin{equation} \label{jabel:eqn19}
		\begin{split}
			\T_1 \(\frac{z}{\tau} | -\frac{1}{\tau} \)&=-i \sqrt{-i\tau} \exp
			(iz^2/(\pi\tau)) \T_1(z\mid\tau), \\
			\T_2
			\(\frac{z}{\tau} | -\frac{1}{\tau} \)&=\sqrt{-i\tau} \exp
			(iz^2/(\pi\tau)) \T_4(z\mid\tau), \\
			\T_4 \(\frac{z}{\tau} | -\frac{1}{\tau} \)&=\sqrt{-i\tau} \exp
			(iz^2/(\pi\tau)) \T_2(z|\tau),\\
			\T_3\(\frac{z}{\tau} | -\frac{1}{\tau} \)&=\sqrt{-i\tau} \exp
			(iz^2/(\pi\tau)) \T_3(z|\tau). 
		\end{split}
	\end{equation}
	In particular, by setting $z=0$ in these formulas, one can conclude  that
	\begin{equation}\label{jabel:eqn20}
		\begin{split}
			\VT'_1 \( -\frac{1}{\tau} \)&=-i \tau \sqrt{-i\tau}~  \VT'_1(
			\tau), \\
			\VT_2 \(-\frac{1}{\tau} \)&=\sqrt{-i\tau}~ \VT_4( \tau),\\
			\VT_4 \( -\frac{1}{\tau} \)&=\sqrt{-i\tau}~
			\VT_2( \tau), \\
			\VT_3 \( -\frac{1}{\tau}
			\)&=\sqrt{-i\tau}~ \VT_3( \tau). 
		\end{split}
	\end{equation}
\end{prop}
Substituting (\ref{jabel:eqn18}) into the first equation in (\ref{jabel:eqn20}) and simplifying, one can
obtain the following well-known modular transformation formula for
$\eta$-function (see also \cite[p.~ 48]{Apostol1990}).
\begin{prop}\label{Dedekind-eta}
	If $\operatorname{Im}(\tau)>0$  and $\sqrt{-i\tau}=+1$ for $\tau=i$, then we have
	\begin{equation}
		\eta\(-\frac{1}{\tau}\)=\sqrt{-\tau i}~ \eta(\tau). \label{jabel:eqn21}
	\end{equation}
\end{prop}
Expand both sides of the first  equation in \reff{jabel:eqn19} into power series about $z$, and then compare the coefficients of $z$ to obtain the following modular transformation formula \cite[p.69]{Apostol1990}
\begin{equation}\label{jabel:eqn22}
	L(-1/{\tau})=-\frac{6\tau i}{\pi}+\tau^2 L(\tau).
\end{equation}

We \cite[Theorem~1.2]{Liu2012JNT} proved the following two remarkable identities for theta functions of degree eight. 
\begin{thm} \label{liuaddthm}
	Suppose that $f(z|\tau)$ is an even entire function of $z$ which satisfies the
	functional equations $f(z|\tau)=f(z+\pi|\tau)=q^4 e^{16iz}f(z+\pi\tau|\tau)$.
	Then we  have
	\begin{align}\label{jabel:eqn23}
		&\frac{4f(x|\tau)}{\T_1^2(2x|\tau)}-
		\frac{4f(y|\tau)}{\T_1^2(2y|\tau)}
		={\T_1(x+y|\tau)\T_1(x-y|\tau)}
		\left\{\frac{-f(0|\tau)}{\T_1^2(x|\tau)\T_1^2(y|\tau)}\right.\\
		&\qquad \qquad \qquad \qquad  +\left. \frac{f(\frac{\pi}{2}|\tau)}{\T_2^2(x|\tau)\T_2^2(y|\tau)}
		-\frac{qf(\frac{\pi+\pi\tau}{2}|\tau)}{\T_3^2(x|\tau)\T_3^2(y|\tau)}+\frac{qf(\frac{\pi\tau}{2}|\tau)}{\T_4^2(x|\tau)\T_4^2(y|\tau)}\right\}.\nonumber
	\end{align}
\end{thm}
Based on the above theta function identity and with the help of asymptotic analysis, we also prove the following theorem.
\begin{thm} \label{liuaddthm:lim}
	Suppose that $f(z|\tau)$ is an even entire function of $z$ which satisfies the
	functional equations $f(z|\tau)=f(z+\pi|\tau)=q^4 e^{16iz}f(z+\pi\tau|\tau)$.
	Then we  have
	\begin{align}\label{jabel:eqn24}
		&\(8L(\tau)+3(\log f)''(0|\tau)\)^2+8M(\tau) +3 (\log f)^{(4)}(0|\tau)\\
		&=\frac{72\VT_1'(\tau)^4}{f(0|\tau)}
		\(\frac{f(\frac{\pi}{2}|\tau)}{\VT_2^4(\tau)}-\frac{qf(\frac{\pi+\pi\tau}{2}|\tau)}{\VT_3^4(\tau)}+\frac{qf(\frac{\pi \tau}{2}|\tau)}{\VT_4^4(\tau)}\). \nonumber
	\end{align}
\end{thm}
These two theorems not only allow us to recover several well-known results in elliptic modular functions, but also lead to several new results in \cite{Liu2012JNT}.  However, many other applications of these two theorems  are not discussed.  In this paper, we will further study the  application of them. In order to show readers the power of these two theorems, we now will give a few applications of these two theorems.

Using Proposition~\ref{doubleperiods} we can verify that the entire function $\T_1^2(2z|\tau)\wp(z|\tau)$ satisfies the conditions of  Theorem~\ref{liuaddthm}. So we can take 
$f(z|\tau)=\T_1^2(2z|\tau)\wp(z|\tau)$ 
in Theorem~\ref{liuaddthm}. By a simple calculation we easily find that  
$\{ 0, \pi/2, (\pi+\pi\tau)/2, (\pi\tau)/2\}$ is a complete set of inequivalent zeros of $\T_1(2z|\tau).$  Using this fact and  the Laurent series expansion of $\wp(z|\tau)$,  we find that
\[
f(\pi/2|\tau)=f((\pi\tau)/2|\tau)=f((\pi+\pi\tau)/2|\tau)=0,
\]
and 
\[
f(0|\tau)=\lim_{z\to 0} \T_1^2(2z|\tau)\wp(z|\tau)=4\VT_1'(\tau)^2.
\]
Substituting the above equations into \reff{jabel:eqn23} we immediately find that  
\begin{equation}\label{jabel:eqn25}
	\wp(x|\tau)-\wp(y|\tau)=-\VT_1'(\tau)^2 \frac{\T_1(x+y|\tau)\T_1(x-y|\tau)}
	{\T_1^2(x|\tau)\T_1^2(y|\tau)}.
\end{equation}
The above formula is equivalent to the addition formula for the Weierstrass sigma-function \cite[p.~179]{Daniels}, \cite[p.~13, Eq.(1)]{Schwarz1893} and \cite[p.~451, Example~ 1]{WhiWat}.
Therefore, Theorem~\ref{liuaddthm}  is indeed a generalization of the addition formula for the Weierstrass sigma-function.

Dividing both sides of \reff{jabel:eqn25} by $y-x$ and then letting $y\to x$, we arrive at the following identity which is equivalent to the  Weierstrass  identity \cite[p.14, Eq.(16)]{Schwarz1893}:
\begin{equation}\label{jabel:eqn26}
	\wp'(x|\tau)=-\VT_1'(\tau)^3 \frac{\T_1(2x|\tau)}{\T_1^4(x|\tau)}.
\end{equation}
We use $e_1(\tau), e_2(\tau)$ and $e_3(\tau)$ to denote the values of $\wp(z|\tau)$ at the half-periods, namely, 
\begin{equation}\label{jabel:eqn27}
	e_1(\tau)=\wp\(\frac{\pi}{2}|\tau\),~e_2(\tau)=\wp\(\frac{\pi \tau}{2}|\tau\), ~e_3(\tau)=\wp\(\frac{\pi+\pi\tau}{2}|\tau\).
\end{equation}
Theorem~\ref{liuaddthm} provides the impetus for perhaps the most straight-forward proof of the following well-known result due to Weierstrass  \cite[Theorem~1.14]{Apostol1990} and \cite[p.~12, Eq.(17)]{Schwarz1893}, which is one of the most fundamental properties of  Weierstrass elliptic functions.
\begin{prop}[Weierstrass]\label{W-pdiff} Let $e_1(\tau), e_2(\tau)$ and $e_3(\tau)$ be defined by \reff{jabel:eqn27}. Then we have 
	\begin{equation}\label{jabel:eqn28}
		\wp'(z|\tau)^2 =4\(\wp(z|\tau)-e_1(\tau)\)\(\wp(z|\tau)-e_2(\tau)\)
		\(\wp(z|\tau)-e_3(\tau)\).
	\end{equation}	
\end{prop}
\begin{proof} Noting that $\wp(z|\tau)$ is an elliptic function which has only one inequivalent pole at $z=0$, of order two,  and $z=0$ is a zero of $\T_1(z|\tau)$,  we find  that the function
	\[
	\(\wp(z|\tau)-e_1(\tau)\)\(\wp(z|\tau)-e_2(\tau)\)\(\wp(z|\tau)-e_3(\tau)\) \T_1^8(z|\tau)
	\]	
	is an entire function of $z$.  Using Proposition~\ref{doubleperiods} we can verify that the above function  satisfies the conditions of  Theorem~\ref{liuaddthm}. So we can take $f(z|\tau)$ as the above function in the theorem. It is easily seen that 
	\[
	f(0|\tau)=f(\pi/2|\tau)=f((\pi\tau)/2|\tau)=f((\pi+\pi\tau)/2|\tau)=0.
	\]
	Substituting these values of $f$ into \reff{jabel:eqn23} we immediately find that 
	\begin{align*}
		&\(\wp(x|\tau)-e_1(\tau)\)\(\wp(x|\tau)-e_2(\tau)\)\(\wp(x|\tau)-e_3(\tau)\) \frac{\T_1^8(x|\tau)}{\T_1^2(2x|\tau)}\\
		&=\(\wp(y|\tau)-e_1(\tau)\)\(\wp(y|\tau)-e_2(\tau)\)\(\wp(y|\tau)-e_3(\tau)\) \frac{\T_1^8(y|\tau)}{\T_1^2(2y|\tau)}.
	\end{align*}
	When $y$ approaches zero, the limit value of the right-hand of the above equatio is $\VT'_1(\tau)^6/4$. Thus we have 
	\begin{align*}
		4\(\wp(x|\tau)-e_1(\tau)\)\(\wp(x|\tau)-e_2(\tau)\)\(\wp(x|\tau)-e_3(\tau)\)=\VT_1'(\tau)^6 \frac{\T_1^2(2x|\tau)}{\T_1^8(x|\tau)}.
	\end{align*}
	Substituting \reff{jabel:eqn26} into the right-hand side of the above equation and replacing $x$ by $z$ we complete the proof of Proposition~\ref{W-pdiff}.		
\end{proof}
Combining\reff{jabel:eqn14} and \reff{jabel:eqn15} yields the Laurent series expansion for $\wp(z|\tau)$ at $z=0,$
\begin{equation}\label{jabel:eqn29}
	\wp(z|\tau)=\frac{1}{z^2}+\frac{1}{15}L(\tau)z^2+\frac{2}{189}M(\tau)z^4+O(z^6),
\end{equation}
 where $L(\tau), M(\tau)$ and $N(\tau)$ are defined by \reff{jabel:eqn9}.

Appealing to this Laurent expansion and using the method of eliminating poles (see, for example \cite[pp.10-11]{Apostol1990}), we can obtain the following differential equation satisfied by the Weierstrass elliptic function $\wp(z|\tau)$ which is equivalent to \cite[Theorem~1.12]{Apostol1990} and 
\cite[p.12, Eq.(14)]{Schwarz1893}.

\begin{prop}[Weierstrass]\label{pdiff}The Weierstrass elliptic function $\wp(z|\tau)$ satisfies the differential equation
	\begin{equation}\label{jabel:eqn30}
		\wp'(z|\tau)^2=4\wp^3(z|\tau)-\frac{4}{3}M(\tau)\wp(z|\tau)-\frac{8}{27}N(\tau).
	\end{equation}	
	\end{prop}

\begin{prop}\label{Jacobiabstruse} Suppose that $u_1, u_2, u_3$ and $u_4$ are complex numbers such that $u_1+u_2+u_3+u_4$ is an integral multiple of $\pi$. Then we have
	\begin{equation}\label{jabel:eqn31}
		\prod_{k=1}^4 \T_2(u_k|\tau)+\prod_{k=1}^4 \T_4(u_k|\tau)=\prod_{k=1}^4 \T_1(u_k|\tau)+\prod_{k=1}^4 \T_3(u_k|\tau).
	\end{equation}	
\end{prop}
\begin{proof} With the help of Proposition~\ref{doubleperiods}, in  Theorem~\ref{liuaddthm}  we can take the entire function $f(z|\tau)$ as
	\[
	f(z|\tau)=\T_1(z-x|\tau)\T_1(z+x|\tau)\T_1(z-y|\tau)\T_1(z+y|\tau)
	\prod_{k=1}^4\T_1(z+u_k|\tau).
	\] 
	It is obvious that  $f(x|\tau)=f(y|\tau)=0$ and appealing to Proposition~\ref{halfperiods} and a direct computation we find that
	\begin{align*}
		&f(0|\tau)=\T_1^2(x|\tau)\T_1^2(y|\tau)\prod_{k=1}^4 \T_1(u_k|\tau),~
		f\(\frac{\pi+\pi\tau}{2}|\tau \)=q^{-1}\T_3^2(x|\tau)\T_3^2(y|\tau)\prod_{k=1}^4 \T_3(u_k|\tau),\\ 
		&f\(\frac{\pi}{2}|\tau\)=\T_2^2(x|\tau)\T_2^2(y|\tau)\prod_{k=1}^4 \T_2(u_k|\tau),~
		f\(\frac{\pi \tau}{2}|\tau\)=q^{-1}\T_4^2(x|\tau)\T_4^2(y|\tau)\prod_{k=1}^4 \T_4(u_k|\tau).
	\end{align*}
	Substituting the above values of $f$ into Theorem~\ref{liuaddthm}, we immediately arrive at \reff{jabel:eqn31}. This completes the proof of Proposition~\ref{Jacobiabstruse}.
\end{proof}

When  $u_1+u_2+u_3+u_4=0$,  Proposition~\ref{Jacobiabstruse} reduces  to  \cite[Theorem~3]{Liuresidue2001}. If we specialize \reff{jabel:eqn31} to the case when $u_1=u_2=u_3=u_4=0$, we obtain Jacobi's quartic theta function identity \cite[p.467]{WhiWat}
\begin{equation}\label{jabel:eqn32}
	\VT_2^4(\tau)+\VT_4^4(\tau)=\VT_3^4(\tau).
\end{equation}
Therefore, we can also think that Theorem~\ref{liuaddthm} is a generalization of  the above identity due to Jacobi. 

 For any integer $a$ and any positive odd integer $n$, we use  $\(\frac{a}{n}\)$ to denote the Jacobi symbol modulo $n$.  The Glaisher--Ramanujan Eisenstein series $a(\tau)$ which is  defined by
\begin{equation}\label{jabel:eqn33}
	a(\tau)=1+6\sum_{n=1}^\infty \(\frac{n}{3}\)\frac{q^n}{1-q^n}.
\end{equation}
J. W. Glaisher \cite{Glaisher1889} studied some arithmetic properties of $a(\tau)$ in 1889, and it was also discussed by Ramanujan in one of his letter to Hardy, written from the nursing home,  Fitzroy House \cite[p.93]{Ramanujan1988}.
It is easily seen that
\begin{equation}\label{jabel:eqn34} a(\tau)=\sqrt{3} \(\log \T_1\)'\(\frac{\pi}{3}\Big|\tau\)  ~\text{and}\quad a(\tau)=-2+3i \(\log \T_1\)'(\pi\tau|3\tau).
\end{equation}

If $k\ge 1$ is a positive integer, we use $r_k(n)$ to denote the number of representations of $n$ as a sum of $k$ squares. We also use $ t_k(n)$
to denote the number of representations of  $n$ as a sum of $k$ triangular numbers.
Following Ramanujan we define the theta functions  $\phi(q)$ and $\psi(q)$ by 
\begin{equation}\label{jabel:eqn35}
	\phi(q)=\sum_{n=-\infty}^\infty q^{n^2} \quad \text{and}\quad \psi(q)=\sum_{n=0}^\infty q^{n(n+1)/2}.
\end{equation}
Consequently, the generating functions for $r_k(n)$ and $t_k(n)$ are given by 
\begin{equation}\label{jabel:eqn36}
	\phi^k(q)=\sum_{n=0}^\infty r_k(n) q^n \quad \text{and}\quad 
	\psi^k(q)=\sum_{n=0}^\infty t_k(n) q^n.
\end{equation}
Using the infinite product representations of $\T_3$ and $\T_2$, one can easily deduce that (see, for example \cite[Corollary~1.3.4]{Berndt06})
\begin{equation}\label{jabel:eqn37}
	\phi(q)=\prod_{n=1}^\infty  (1-q^{2n}) (1+q^{2n-1})^2 \quad \text{and}\quad 
	\psi(q)=\prod_{n=1}^\infty (1-q^n) (1+q^n)^2.
\end{equation}

In this paper we also need the trigonometric series expansions for the partial logarithmic derivatives of $\T_2, \T_3$ and $\T_4$, which are given by (see \cite[p.489]{WhiWat})
\begin{equation}\label{jabel:eqn38}
	\begin{split}
		(\log \T_2)'(z|\tau)&=-\tan z+4\sum_{n=1}^\infty \frac{(-q)^n}{1-q^n} \sin 2nz,\\
		(\log \T_4)'(z|\tau)&=4\sum_{n=1}^\infty \frac{q^{n/2}}{1-q^n} \sin 2nz,\\
		(\log \T_3)'(z|\tau)&=4\sum_{n=1}^\infty (-1)^n \frac{q^{n/2}}{1-q^n} \sin 2nz.
	\end{split}
\end{equation}
The rest of the paper is organized as follows.  In Section~2 we will use Theorem~\ref{liuaddthm} to prove a general theta function identity of degree $3$ and give a few  applications of this identity.

 In Section~3 we will use Theorem~\ref{liuaddthm} to  investigate a generalization of the Kiepert quintuple product identity and its application.
In particular we prove the following proposition.
\begin{prop}\label{dirichletgaussKiepert} If $m$ is square-free such that
	$m\equiv 1 \pmod 4$, then we have
	\begin{equation}\label{jabel:eqn43}
		\(\prod_{n=1}^\infty (1-q^n)\)  \sum_{k=1}^{\frac{m-1}{2}}\(\frac{k}{m}\)
		\frac{\T_1(\frac{4k\pi}{m}|\tau)}{\T_1(\frac{2k\pi}{m}|\tau)}
		=\sqrt{m} \sum_{n=-\infty}^\infty (-1)^n \(\frac{6n+1}{m}\)q^{n(3n+1)/2},
	\end{equation}
where $\(\frac{k}{m}\)$ is the Jacobi symbol.
\end{prop}
In Section~4 we will use Theorem~\ref{liuaddthm} to prove the following addition formula and give a few applications. 
\begin{thm}  \label{6RRCthm:n1}  Suppose that $F(z|\tau)$ and $G(z|\tau)$ are two entire functions of degree $6$, which satisfy the functional equations 
	\begin{align} \label{rrc:eqn1}
		F(z|\tau)=F(z+\pi|\tau)=q^{3} e^{12iz}F(z+\pi\tau|\tau)
	\end{align}	
	and 
	\begin{align}\label{rrc:eqn2}
		G(z|\tau)=G(z+\pi|\tau)=q^{3} e^{12iz}G(z+\pi\tau|\tau).
	\end{align}	
	 Then there exists a constant $C$ independent of $x$ and $y$ such that
	\begin{align}\label{rrc:eqn3}
		&\(F(x|\tau)-F(-x|\tau)\)\(G(y|\tau)-G(-y|\tau)\)\\
		&-\(F(y|\tau)-F(-y|\tau)\)\(G(x|\tau)-G(-x|\tau)\)\nonumber\\
		&=C\T_1(x-y|\tau)\T_1(x+y|\tau)\T_1(2x|\tau)\T_1(2y|2\tau).\nonumber
	\end{align}		
	\end{thm}
In Section~5 we will discuss some applications of a general theta function identity of degree $6$. In particular, we prove that
\begin{equation*}
	\frac{\VT_2(21\tau)}{\VT_2(\tau)}-\frac{\VT_3(21\tau)}{\VT_3(\tau)}	
	+\frac{\VT_4(21\tau)}{\VT_4(\tau)}
	=\frac{4\eta^3(3\tau)\eta(7\tau)}{\eta^4(\tau)}-7\frac{\eta^3(21\tau)}{\eta^3(\tau)}.
\end{equation*}
In Section~6 we will investigate the application of Theorem~\ref{liuaddthm:lim} to Eisenstein series identities. For example,  we find that
\begin{align*}
	&(9L(9\tau)-L(\tau))^2 +\frac{1}{5}\(42M(9\tau)-2M(\tau)\)\\
	&=\frac{72\VT_1'(9\tau)^5}{\VT_1'(\tau)}
	\(\frac{\VT_2(\tau)}{\VT_2^5(9\tau)}-\frac{\VT_3(\tau)}{\VT_3^5(9\tau)}+\frac{\VT_4(\tau)}{\VT_4^5(9\tau)}\),\nonumber
\end{align*}
 where $L(\tau)$ and $M(\tau)$ are the first two Eisenstein series defined in \reff{jabel:eqn9}.	

More applications of Theorem~\ref{liuaddthm} to modular function identities are given in Section~7.
%%%%%%%%%%%%%%%%%%%%%%%%%%%%%%%%%%%%%%%%%%%%%%%%%%%%%%%%%%%%%%%%%%%%%%%%%
\section{A general theta function identity of degree $3$}
%%%%%%%%%%%%%%%%%%%%%%%%%%%%%%%%%%%%%%%%%%%%%%%%%%%%%%%%%%%%%%%%%%%%%%%%%
\begin{thm}  \label{degree3addthm}  Suppose that $F(z|\tau)$ and $G(z|\tau)$ are two odd entire functions of degree $3$, which satisfy the functional equations 
	\begin{align} \label{liu:eqn1}
		F(z|\tau)=-F(z+\pi|\tau)=-q^{3/2} e^{6iz}F(z+\pi\tau|\tau)
	\end{align}	
	and 
	\begin{align}\label{liu:eqn2}
		G(z|\tau)=-G(z+\pi|\tau)=-q^{3/2} e^{6iz}G(z+\pi\tau|\tau).
	\end{align}	
	Then there exists a constant $C$ independent of $x$ and $y$ such that
	\begin{align}\label{liu:eqn3}
		F(x|\tau) G(y|\tau)-F(y|\tau)G(x|\tau)
		=C\T_1(x-y|\tau)\T_1(x+y|\tau)\T_1(x|\tau)\T_1(y|\tau).
	\end{align}		
\end{thm}
\begin{proof}
Let $F(z|\tau)$ and $G(z|\tau)$ be the two entire functions given in Theorem~\ref{degree3addthm}. Then
\[
f(z|\tau)=\frac{\(F(z|\tau)G(y|\tau)-G(z|\tau)F(y|\tau)\)\T_1^2(2z|\tau)}{\T_1(z-y|\tau)\T_1(z+y|\tau)\T_1(z|\tau)}
\]	
satisfies the conditions of Theorem~\ref{liuaddthm}. Appealing to L'Hospital's rule and a simple calculation we find that
\[
f(y|\tau)=\frac{\(F'(y|\tau)G(y|\tau)-G'(y|\tau)F(y|\tau)\)\T_1(2y|\tau)}{\VT'_1(\tau)\T_1(y|\tau)}.
\]
Noting that $0, \pi/2, (\pi+\pi\tau)/2$ and $(\pi \tau)/2$ are zeros of $\T_1(2z|\tau)$, we immediately find that
\[
f(0|\tau)=f(\pi/2|\tau)=f((\pi+\pi\tau)/2|\tau)=f((\pi\tau)/2|\tau)=0.
\]
Substituting these values of $f$ into \reff{jabel:eqn23} and simplifying  we easily conclude that
\begin{align*}
&\frac{F(x|\tau)G(y|\tau)-F(y|\tau)G(x|\tau)}{\T_1(x-y|\tau)\T_1(x+y|\tau)\T_1(x|\tau)\T_1(y|\tau)}\\
&=\frac{F'(y|\tau)G(y|\tau)-G'(y|\tau)F(y|\tau)} {\VT_1'(\tau)\T_1^2(y|\tau)\T_1(2y|\tau)}.	
\end{align*}
The right-hand side of the above equation is independent of $x$, so also does the left-hand side.  It is obvious that  the left-hand  side of the above equation is symmetric about $x$ and $y$, so the left-hand side  of the above equation is also independent of $y$. Thus there exists a constant $C$ independent of $x$ and $y$ such that
\[
\frac{F(x|\tau)G(y|\tau)-F(y|\tau)G(x|\tau)}{\T_1(x-y|\tau)\T_1(x+y|\tau)\T_1(x|\tau)\T_1(y|\tau)}=C,
\]
 which is equivalent to \reff{liu:eqn3}. We thus complete the proof of Theorem~\ref{degree3addthm}.
\end{proof}
Theorem~\ref{degree3addthm} is equivalent to \cite[Theorem~1]{Liu2007ADV}, which  has been used in \cite{Liu2007ADV, Liu2009pac} to derive many elliptic function identities, including Ramanujan's cubic theta function identity and Winquist's identity \cite{Winquist1969}.  Next we will give a few applications of Theorem~\ref{degree3addthm}.

By taking $F(z|\tau)=\T_1(3z|3\tau)$ and 
\[
G(z|\tau)=\( \frac{\sin z}{\sin 3z}+2\sum_{n=1}^\infty \(\frac{n}{3}\)
\frac{q^n}{1-q^n} \cos 2nz \) \T_1(3z|3\tau)
\]
in Theorem~\ref{degree3addthm} and simplifying we arrive at the following Lambert series identity \cite[Eq.(3.20)]{Liu2010IMRN}:
\begin{align}\label{liu:eqn4}
	\frac{\sin x}{ \sin 3x}-\frac{\sin y}{\sin 3y}+	2\sum_{n=1}^\infty \(\frac{n}{3}\)
	\frac{q^n}{1-q^n} \(\cos 2nx-\cos 2ny\)\\
	=\frac{\eta^3(3\tau)\T_1(x|\tau)\T_1(y|\tau)\T_1(x-y|\tau)\T_1(x+y|\tau)}{\eta^3(\tau)\T_1(3x|3\tau)\T_1(3y|3\tau)}.\nonumber
\end{align}
Theorem~\ref{degree3addthm} can be used to derive the following general Lambert series identity.
\begin{prop}\label{degree3addthm:a} The following Lambert series identity related to theta functions holds: 
	\begin{align} \label{liu:eqn5}
		&\sum_{n=1}^\infty \frac{q^{n/2}}{1-q^n} (\cos 2nx- \cos 2ny) \sin (2nu)\\
		&=-\frac{\eta^3(\tau)\T_1 (2u|\tau)\T_1(x+y|\tau)\T_1(x-y|\tau)}{4\T_4(x+u|\tau)\T_4(x-u|\tau)\T_4(y+u|\tau)\T_4(y-u|\tau)}.\nonumber
	\end{align}
\end{prop}
\begin{proof} With the help of Proposition~\ref{doubleperiods} we can verify that $F(z|\tau)$ and $G(z|\tau)$ satisfies the conditions of Theorem~\ref{degree3addthm}, where $F(z|\tau)$ and $G(z|\tau)$ are given by 
	\[
	F(z|\tau)=\T_4(z+u|\tau)\T_4(z-u|\tau)\T_1(z|\tau),
	\]
	and
	\[
	G(z|\tau)=\T_1(z|\tau)\T_4(z+u|\tau)\T_4(z-u|\tau)\((\log \T_4)'(z+u|\tau)-(\log \T_4)'(z-u|\tau)\).
	\]
	Thus there exists a constant $C$ independent of $x$ and $y$ such that	
    \begin{equation*}
    	(\log \T_4)'(x+u|\tau)-(\log \T_4)'(x-u|\tau)-(\log \T_4)'(y+u|\tau)+(\log \T_4)'(y-u|\tau)
    \end{equation*}
    \begin{equation*}
    	=\frac{C\T_1(x-y|\tau)\T_1(x+y|\tau)}{\T_4(x+u|\tau)\T_4(x-u|\tau)\T_4(y+u|\tau)\T_4(y-u|\tau)}.
    \end{equation*}	
Multiplying both sides of the above equation by $\T_4(y-u|\tau)$, letting $y\to u+(\pi\tau/2)$ and simplifying we find that $C=-\VT_1'(\tau)\T_1(2u|\tau)$. It follows that
 \begin{equation*}
	(\log \T_4)'(x+u|\tau)-(\log \T_4)'(x-u|\tau)-(\log \T_4)'(y+u|\tau)+(\log \T_4)'(y-u|\tau)
\end{equation*}
\begin{equation*}
	=\frac{-\VT_1'(\tau)\T_1(x-y|\tau)\T_1(x+y|\tau)}{\T_4(x+u|\tau)\T_4(x-u|\tau)\T_4(y+u|\tau)\T_4(y-u|\tau)\T_1(2u|\tau)}.
\end{equation*}
Substituting the trigonometric series expansion for $(\log \T_4)'(z|\tau)$ into the left-hand side of the above equation and simplifying we arrive at 
\reff{liu:eqn5}. This completes the proof of Proposition~\ref{degree3addthm:a}.	
\end{proof}
For any rational integers $m$ and $n$, we will use $\(\frac{m}{n}\)$ to denote the Kronecker symbol.

By setting $u=\pi/4$ in \reff{liu:eqn5} and noting that $\sin \frac{\pi n}{2}=(\frac{-4}{n})$,  we conclude that
\begin{equation}\label{liu:eqn6}
	4\sum_{n=0}^\infty \(\frac{-4}{n}\) \frac{q^{n/2}}{1-q^n} \(\cos 2nx -\cos 2ny\) 
	=-\VT_2^2(\tau)\frac{\T_1(x+y|\tau)\T_1(x-y|\tau)}{\T_3(2x|2\tau)\T_3(2y|2\tau)}.
\end{equation}

Replaing $\tau$ by $3\tau$ in \reff{liu:eqn5} and then setting $u=(\pi\tau/2)$ and simplifying we arrive at \cite[Eq.(3.16)]{Liu2010IMRN}
\begin{align}\label{liu:eqn7}
	&\sum_{n=1}^\infty \frac{q^n}{1+q^n+q^{2n}} \(\cos 2nx-\cos 2ny\)\\
	&=-\frac{\eta^3(\tau)\T_1(x|3\tau)\T_1(y|3\tau)\T_1(x+y|3\tau)\T_1(x-y|3\tau)}{2\eta^3(3\tau)\T_1(x|\tau)\T_1(y|\tau)}.\nonumber
\end{align}

By choosing $F(z|\tau)=\T_1(z|\tau)\T_2(2z|2\tau)$ and 
\[
G(z|\tau)=\(\frac{1}{\cos 2z}+4 \sum_{n=1}^\infty \(\frac{-4}{n}\)\frac{q^n}{1-q^n} \cos 2nz\)\T_1(z|\tau)\T_2(2z|2\tau),
\]
in Theorem~\ref{degree3addthm} and making some simple calculations we find that 
\begin{align}\label{liu:eqn8}
	&\frac{1}{\cos 2x}-\frac{1}{\cos 2y} +4\sum_{n=1}^\infty \(\frac{-4}{n}\) \frac{q^n}{1-q^n} \(\cos 2nx -\cos 2ny\)\\
	&=\VT_2^2(\tau) \frac{\T_1(x+y|\tau)\T_1(x-y|\tau)}{2\T_2(2x|2\tau)\T_2(2y|2\tau)}.\nonumber
\end{align}

In exactly the same way, if we take $F(z|\tau)=\T_1^3(z+\frac{\pi}{3}|\tau)+\T_1^3(z-\frac{\pi}{3}|\tau)$ and $G(z|\tau)=\T_1(3z|3\tau)$ in Theorem~\ref{degree3addthm}, after a little reduction, we deduce that \cite[Eq.(3.43)]{Liu2007ADV}
\begin{align}\label{liu:eqn9}
	&\(\T_1^3(x+\frac{\pi}{3}|\tau)+\T_1^3(x-\frac{\pi}{3}|\tau)\)\T_1(3y|3\tau)
	\\
	&-\(\T_1^3(y+\frac{\pi}{3}|\tau)+\T_1^3(y-\frac{\pi}{3}|\tau)\)\T_1(3x|3\tau)\nonumber\\
	&=\frac{3\eta^3(3\tau)}{\eta^3(\tau)}\T_1(x|\tau)\T_1(y|\tau)\T_1(x+y|\tau)\T_1(x-y|\tau).\nonumber
\end{align}
Dividing both sides of the above equation by $y$ and then letting $y \to 0$, we conclude that
\begin{equation}\label{liu:eqn10}
	\T_1^3(x+\frac{\pi}{3}|\tau)+\T_1^3(x-\frac{\pi}{3}|\tau)-\T_1^3(x|\tau)=3a(\tau)\T_1(3x|3\tau),
\end{equation}
where $a(\tau)$  is the Glaisher--Ramanujan Eisenstein series given by \reff{jabel:eqn33}.

 By taking $F(z|\tau)={\T_1(3z|3\tau)}$ and $G(z|\tau)=\T_1(z|\tau/3)$ in Theorem~\ref{degree3addthm} and making a simple calculation, we deduce that
\cite[Theorem~5]{LiuTrans}
\begin{align}\label{liu:eqn11}
	&\eta^2(\tau)\T_1(3y|3\tau)\T_1\(x|\frac{\tau}{3}\)-\eta^2(\tau)\T_1(3x|3\tau)\T_1\(y|\frac{\tau}{3}\)\\
	&=\T_1(x|\tau)\T_1(y|\tau)\T_1(x-y|\tau)\T_1(x+y|\tau).\nonumber
\end{align}	
This identity allows us to derive the following identity (see \cite[pp~.829--830]{LiuTrans} for details):
\begin{align}\label{liu:eqn12}
	&32\prod_{n=1}^\infty (1-q^n)^{10}\\
	&=9\(\sum_{n=-\infty}^\infty(-1)^n (2n+1)^3 q^{3n(n+1)/2}\)\(\sum_{n=-\infty}^\infty(-1)^n (2n+1) q^{n(n+1)/6}\)\nonumber\\
	&-\(\sum_{n=-\infty}^\infty(-1)^n (2n+1) q^{3n(n+1)/2}\)\(\sum_{n=-\infty}^\infty(-1)^n (2n+1)^3 q^{n(n+1)/6}\).\nonumber
\end{align}
A short proof of Ramanujan’s partition congruence for the modulus $11$, $p(11n+6)\equiv 0 \pmod {11}$ is given in \cite{BCLY}  with the help of 
\reff{liu:eqn11}.

By taking $F(z|\tau)=\T_1(z|\tau)\T_1(z-u|\tau)\T_1(z+u|\tau)$ 
and $G(z|\tau)=\T_1(z|\tau)\T_1(z-v|\tau)\T_1(z+v|\tau)$ 
in Theorem~\ref{degree3addthm} we deduce that
\begin{align}\label{liu:eqn12}
	&\T_1(x-u|\tau)\T_1(x+u|\tau)\T_1(y-v|\tau)\T_1(y+v|\tau)\\
	&~ -\T_1(y-u|\tau)\T_1(y+u|\tau)\T_1(x-v|\tau)\T_1(x+v|\tau)\nonumber\\
	&=\T_1(u+v|\tau)\T_1(u-v|\tau)\T_1(x-y|\tau)\T_1(x+y|\tau).\nonumber
\end{align}	
This identity is equivalent to the Weierstrass three-term theta function identity \cite{Weierstrass1882} (see,  also \cite[p.~ 451,  Example~5]{WhiWat}, \cite[p.~142, Eq.(25)]{Enne 1890}, \cite[Eq.(1.1)]{Koornwinder2014}, \cite[Theorem~7]{Liu2007ADV}).
%%%%%%%%%%%%%%%%%%%%%%%%%%%%%%%%%%%%%%%%%%%%%%%%%%%%%%%%%%%%%%%%%%%%%%%%%%
\section{A generalization of the Kiepert quintuple product identity and its application}
%%%%%%%%%%%%%%%%%%%%%%%%%%%%%%%%%%%%%%%%%%%%%%%%%%%%%%%%%%%%%%%%%%%%%%%%%%%
In this section we will first use Theorem~\ref{liuaddthm} to derive  the following theta function identities of degree $4$, which is equivalent to \cite[Theorem~1]{Liu2005ADV}. 
\begin{thm} \label{KLthm} Suppose that $F(z|\tau)$ is an odd entire function of $z$ which satisfies the functional equations $F(z|\tau)=F(z+\pi|\tau)=q^2 e^{8iz}F(z+\pi\tau|\tau)$. Then we  have
	\begin{equation}\label{KL:eqn1}
		\frac{F(x|\tau)}{\T_1(2x|\tau)}=	\frac{F(y|\tau)}{\T_1(2y|\tau)},
	\end{equation}	
	or
	\begin{equation}\label{KL:eqn1a}
	F(x|\tau)=C\T_1(2x|\tau),	
	\end{equation}
	where $C$ is a constant independent  of $x$.
\end{thm}
\begin{proof} 
	Let $F(z|\tau)$ be the given function in Theorem~\ref{liuaddthm}. Now in Theorem~\ref{liuaddthm} we can take
	\[
	f(z|\tau)=F(z|\tau)\T_1(2z|\tau)
	\]
	since it satisfies all conditions of Theorem~\ref{liuaddthm}. It is obvious that  $\T_1(0|\tau)=0.$ Using this fact and the functional equations $\T_1(z|\tau)=-\T_1(z+\pi|\tau)=-q^{1/2} e^{2iz} \T_1(z+\pi\tau|\tau)$ in Proposition~\ref{doubleperiods}, we deduce that $0, \pi/2, (\pi+\pi\tau)/2$ and $\pi\tau/2$ are zeros of $\T_1(2z|\tau)$. It follows that  $f(0|\tau)=f(\pi/2|\tau)=f((\pi+\pi\tau)/2|\tau)=f(\pi\tau/2|\tau)=0$. Substituting these values of $f(z|\tau)$ into Theorem~\ref{liuaddthm} we arrive at \reff{KL:eqn1}.  Equation \reff{KL:eqn1} indicates that 
	$F(x|\tau)/{\T_1(2x|\tau)}$ is independent of $x$, and so it must be a constant, say $C$. Hence we obtain \reff{KL:eqn1a}.  This  completes the proof of Theorem~\ref{KLthm}.
\end{proof}
This beautiful formula has a lot of applications in number theory. Let's give some examples.
%%%%%%%%%%%%%%%%%%%%%%%%%%%%%%%%%%%%%%%%%%%%%%%%%%%%%%%%%%%%%%%%%%%%%%%%%%%%%
\subsection{The Kiepert quintuple product identity}
%%%%%%%%%%%%%%%%%%%%%%%%%%%%%%%%%%%%%%%%%%%%%%%%%%%%%%%%%%%%%%%%%%%%%%%%%%%
The quintuple product identity was first discovered by Kiepert \cite[p.213]{Kiepert1879}, and later rediscovered by others many times.  A survey of the quintuple product, which includes 29 proofs of this beautiful identity, has been given by Cooper \cite{Cooper2006}. Next we will use Theorem~\ref{KLthm} to  provide a proof of  the  Kiepert quintuple product identity, and  the proof is slightly different from that of \cite{Liu2005ADV}. For the Kiepert quintuple product identity, one can see also Chan \cite[Theorem~4.3]{Chan2020}. 
\begin{thm}\label{KLthm:a} Kiepert's  quintuple product identity states that
	\begin{equation}\label{KL:eqn2}
		2\sum_{n=-\infty}^\infty (-1)^n q^{n(3n+1)/2} \cos (6n+1) z
		=\(\prod_{n=1}^\infty (1-q^n)\) \frac{\T_1(2z|\tau)}{\T_1(z|\tau)}.
	\end{equation}		
\end{thm}
\begin{proof}
	Using Proposition~\ref{doubleperiods} and a simple calculation we can verify that the odd entire function
\[
F(z|\tau)=\(e^{2iz}\T_1\(3z+\pi\tau|3\tau\)-e^{-2iz}\T_1\(3z-\pi\tau|3\tau\)\)\T_1(z|\tau) 
\]
satisfies the conditions of Theorem~\ref{KLthm}.  Thus there exists a constant $C$ independent of $z$ such that
\[
\(e^{2iz}\T_1\(3z+\pi\tau|3\tau\)-e^{-2iz}\T_1\(3z-\pi\tau|3\tau\)\)\T_1(z|\tau)=C\T_1(2z|\tau).
\]

Putting $z=\pi/3$ in the above equation and noting that $\T_1(\pi\tau|3\tau)=iq^{-1/6}\eta(\tau)$, we find that
\begin{align}\label{KL:eqn4}
	&e^{2iz}\T_1(3z+\pi\tau|3\tau)-e^{-2iz}\T_1(3z-\pi\tau|3\tau)\\
	&=iq^{-1/8} \(\prod_{n=1}^\infty (1-q^n)\)\frac{\T_1(2z|\tau)}{\T_1(z|\tau)}.\nonumber
\end{align}
	Using the infinite series representation of $\T_1$ and a direct computation, we find that
\begin{align}\label{KL:eqn5}
	&e^{2iz}\T_1(3z+\pi\tau|3\tau)-e^{-2iz}\T_1(3z-\pi\tau|3\tau)\\
	&=2iq^{-1/8} \sum_{n=-\infty}^\infty (-1)^n q^{(3n^2+n)/2} \cos (6n+1)z.
	\nonumber
\end{align}
 A comparison of the above equation with \reff{KL:eqn4}  completes the proof of Theorem~\ref{KLthm:a}.
\end{proof}
To discuss an application of Theorem~\ref{KLthm:a}, we need the following proposition due to Dirichlet \cite[p.303]{Dirichlet1894},  which gives the value of Gauss's sum.
\begin{prop}\label{dirichletgauss:n1} If $m$ is square-free and odd,  $h$ is any positive integer, then we have 
	\begin{equation}\label{dg:eqn1}
		\sum_{k=1}^{m-1}\(\frac{k}{m}\) e^{\frac{2kh\pi i}{m}}=\(\frac{h}{m}\) i^{\frac{1}{4}(m-1)^2} \sqrt{m}.
	\end{equation}	
\end{prop}
Appealing to the above proposition and making a simple calculation we can derive the following proposition.
\begin{prop}\label{dirichletgauss:n2} If $m$ is square-free such that
	$m\equiv 1 \pmod 4$	and $h$ is any positive integer, then we have 
	\begin{equation}\label{dg:eqn2}
		\sum_{k=1}^{\frac{m-1}{2}}\(\frac{k}{m}\) \cos \(\frac{2kh\pi }{m}\)=\frac{1}{2}\(\frac{h}{m}\)  \sqrt{m}.
	\end{equation}	
\end{prop}
Using this proposition and the Kiepert quintuple product identity we can prove Proposition~\ref{dirichletgaussKiepert}.

\begin{proof} If we replace $x$ by $(2k\pi)/m$ in \reff{KL:eqn2} and then multiply the resulting equation by $(\frac{k}{m})$, then we deduce that
	\begin{align*}
		&2\sum_{n=-\infty}^\infty (-1)^n q^{n(3n+1)/2}\(\frac{k}{m}\) \cos \(\frac{2(6n+1)k \pi}{m}\)\\
		&=\(\prod_{n=1}^\infty (1-q^n)\) \(\frac{k}{m}\)
		\frac{\T_1(\frac{4k\pi}{m}|\tau)}{\T_1(\frac{2k\pi}{m}|\tau)}.
	\end{align*}
	Putting $k=1, 2, \ldots, {m-1}$ respectively in the above equation and then summing the resulting equations, we conclude that
	\begin{align}\label{dg:eqn3}
		&2\sum_{n=-\infty}^\infty (-1)^n q^{n(3n+1)/2}\sum_{k=1}^{\frac{m-1}{2}} \(\frac{k}{m}\) \cos \(\frac{2(6n+1)k \pi}{m}\)\\
		&=\(\prod_{n=1}^\infty (1-q^n)\) \sum_{k=1}^{\frac{m-1}{2}} \(\frac{k}{m}\)
		\frac{\T_1(\frac{4k\pi}{m}|\tau)}{\T_1(\frac{2k\pi}{m}|\tau)}.\nonumber
	\end{align}	
	With the help of Proposition~\ref{dirichletgauss:n2}, we find that
	\[
	\sum_{k=1}^{\frac{m-1}{2}} \(\frac{k}{m}\) \cos \(\frac{2(6n+1)k \pi}{m}\)
	=\frac{1}{2}\(\frac{6n+1}{m}\) \sqrt{m}.
	\]
	Substituting the above equation into the left-hand side of \reff{dg:eqn3}, we complete the proof of Proposition~\ref{dirichletgaussKiepert}.	
\end{proof}
%%%%%%%%%%%%%%%%%%%%%%%%%%%%%%%%%%%%%%%%%%%%%%%%%%%%%%%%%%%%%%%%%%%%%%%
\subsection{A new form of the quintuple product identity} 
The following proposition  provides a new  form for the Kiepert quintuple product identity.
\begin{prop}\label{4degreethm:sp4} Let $(\frac{m}{n})$  be the Kronecker symbol. Then we have 
	\begin{equation} \label{newf:eqn1}
	2\sum_{n=1}^\infty \(\frac{12}{n}\) q^{\frac{n^2}{24}} \cos n z	=\eta(\tau)\frac{\T_1(2z|\tau)}{\T_1(z|\tau)}. 
	\end{equation}		
\end{prop}

\begin{proof}  With the help of Proposition~\ref{doubleperiods} it is easy to   verify that the entire function
	\[
	F(z|\tau)=\(\T_1\(z+\frac{2\pi}{3}|\frac{\tau}{3}\)-\T_1\(z-\frac{2\pi}{3}|\frac{\tau}{3}\)\)\T_1(z|\tau)
	\]
	satisfies the conditions of Theorem~\ref{KLthm}.  Hence there exists a constant $C$ independent of $z$ such that
	\[
	\(\T_1\(z+\frac{2\pi}{3}|\frac{\tau}{3}\)-\T_1\(z-\frac{2\pi}{3}|\frac{\tau}{3}\)\)\T_1(z|\tau)=C\T_1(2z|\tau).
	\]
	
	Setting $y=\pi/3$ in the above equation and noting that 
	$\T_1(\frac{\pi}{3}|\frac{\tau}{3})=\sqrt{3}\eta(\tau)$ we conclude that
	\begin{equation}\label{newf:eqn2}
	\T_1\(z+\frac{2\pi}{3}|\frac{\tau}{3}\)-\T_1\(z-\frac{2\pi}{3}|\frac{\tau}{3}\)=\sqrt{3}\eta(\tau)\frac{\T_1(2z|\tau)}{\T_1(z|\tau)}.
	\end{equation}
	Using the series representation of $\T_1(z|\tau)$ in Definition~\ref{jtheta} and a simple calculation, we have 
	\begin{align}\label{newf:eqn3}
		\T_1\(z+\frac{2\pi}{3}|\frac{\tau}{3}\)-\T_1\(z-\frac{2\pi}{3}|\frac{\tau}{3}\)=2\sqrt{3}\sum_{n=1}^\infty \(\frac{12}{n}\) q^{\frac{n^2}{24}} \cos n z.\nonumber
	\end{align}
Combining the above two equations, we complete the proof of Proposition~\ref{4degreethm:sp4}.
\end{proof}
It can be showed  that Proposition~\ref{4degreethm:sp4} is equivalent to the Kiepert quintuple product identity in Theorem~\ref{KLthm:a}. I think that Proposition~\ref{4degreethm:sp4} is more beautiful in form than  Theorem~\ref{KLthm:a}.
Setting $z=0$ in Proposition~\ref{4degreethm:sp4}, we immediately obtain 
the following proposition (see, for example \cite[p.(xiii)]{Kohler2011}).
\begin{cor}[Euler's pentagonal number theorem]\label{KLexmp:n1} 
	\begin{equation}\label{newf:eqn4}
		\eta(\tau)=\sum_{n=1}^\infty \(\frac{12}{n}\)q^{\frac{n^2}{24}}.
	\end{equation}	
	\end{cor}

%%%%%%%%%%%%%%%%%%%%%%%%%%%%%%%%%%%%%%%%%%%%%%%%%%%%%%%%%%%%%%%%%%%%%%%%%%%%%
\subsection{New proofs of Jacobi's two-square theorem and the two-triangular number theorem}

\begin{thm} \label{jacobi2square} There holds the identity
	\begin{equation}\label{jacsquare:eqn4}
		\(\sum_{n=-\infty}^\infty  q^{n^2}\)^2
		=1+4\sum_{n=0}^\infty 
		\(\frac{q^{4n+1}}{1-q^{4n+1}}-\frac{q^{4n+3}}{1-q^{4n+3}}\).
	\end{equation}
	Consequently, the number $r_2(n)$ of representation of the positive integer $n$
	as a sum of two squares  is given by
	\begin{equation}\label{jacsquare:eqn5}
		r_2(n)=4(d_1(n)-d_3(n)), 
	\end{equation}
	where
	\begin{equation}\label{jacsquare:eqn6}
		d_k(n)=\sum_{d|n, d \equiv k \pmod 4}1. 
	\end{equation}	
\end{thm}
This theorem can be found in \cite[Theorem~3.2.1]{Berndt06}, \cite[pp.15-16]{Grosswald1985} and \cite[Eq. (1.2)]{ChanKrattenthaler2005}.

\begin{proof}By taking $F(z|\tau)=\T_1^4(z+\frac{\pi}{4}|\tau)$ in Theorem~\ref{KLthm} we can easily deduce that
	\begin{equation}\label{jacsquare:eqn7}
		\T_1^4\(z+\frac{\pi}{4}|\tau\)-\T_1^4\(z-\frac{\pi}{4}|\tau\)=\VT_2^3(\tau)\T_1(2z|\tau).
	\end{equation}
	Differentiating both sides of the above equation with respect to $z$ and then letting $z=0$, we easily find that	
	\begin{equation}\label{jacsquare:eqn8}
		\(\log \T_1\)'(\frac{\pi}{4}|\tau)=\frac{\VT_2^3(\tau)
			\VT_1'(\tau)}{4\T_1^4\(\frac{\pi}{4}\Big|\tau\)}.
	\end{equation}
	Using the infinite product representations of $\T_1$ and $\T_2$, one can easily find that
	\begin{align*}
		\VT_1'(\tau)&=2q^{1/8}\prod_{n=1}^\infty (1-q^n)^3,\\
		\VT_2(\tau)&=2q^{1/8}\prod_{n=1}^\infty (1-q^n)(1+q^n)^2,\\
		\T_1\(\frac{\pi}{4}\Big|\tau\)&=\sqrt{2}q^{1/8}\prod_{n=1}^\infty (1-q^n)(1+q^{2n}).
	\end{align*}
	Substituting the above three equations into the right-hand side of \reff{jacsquare:eqn8}, we deduce  that
	\begin{equation}\label{jacsquare:eqn9}
		\(\log \T_1\)'(\frac{\pi}{4}|\tau)=\prod_{n=1}^\infty (1-q^{2n})^2 (1+q^{2n-1})^4.
	\end{equation}
	Setting $z=\pi/4$ in \reff{jabel:eqn13} and  by a simple calculation, we easily find that
	\begin{equation}\label{jacsquare:eqn10}
		\(\log \T_1\)'(\frac{\pi}{4}|\tau)	=1+4\sum_{n=0}^\infty 
		\(\frac{q^{4n+1}}{1-q^{4n+1}}-\frac{q^{4n+3}}{1-q^{4n+3}}\).
	\end{equation}
	Combining the above equations we immediately conclude that
	\[
	\prod_{n=1}^\infty (1-q^{2n})^2 (1+q^{2n-1})^4
	=1+4\sum_{n=0}^\infty 
	\(\frac{q^{4n+1}}{1-q^{4n+1}}-\frac{q^{4n+3}}{1-q^{4n+3}}\).
	\]
	Substituting the above equation and the first identity in \reff{jabel:eqn37}  into the left-hand side of the above equation we arrive at \reff{jacsquare:eqn4}. This completes the proof of Theorem~\ref{jacobi2square}.
\end{proof}

\begin{thm}\label{2trigothm} Let $\psi(q)$ be defined by \reff{jabel:eqn35}. Then  we have 
	\begin{equation}\label{jacsquare:eqn11}
		\psi^2(q^2)=\sum_{n=0}^\infty (-1)^n \frac{q^n}{1-q^{2n+1}},
	\end{equation}	
	and thus
	\begin{equation}\label{jacsquare:eqn12}
		t_2(n)=d_1(4n+1)-d_3(4n+1).
	\end{equation}
\end{thm}
This theorem can be found in \cite[p.397, Entry 18.2.4]{ABLost} and \cite[Theorem~7]{Ewell1992}.
\begin{proof}
	By taking $F(z|\tau)=\T_4^4(z+\frac{\pi}{4}|\tau)$ in Theorem~\ref{KLthm} we can easily deduce that
	\begin{equation}\label{jacsquare:eqn13}
		\T_4^4\(z+\frac{\pi}{4}|\tau\)-\T_4^4\(z-\frac{\pi}{4}|\tau\)=\VT_2^3(\tau)\T_1(2z|\tau).
	\end{equation}	
	Differentiating both sides of the above equation with respect to $z$ and then setting $z=0$ we conclude that 
	\begin{equation}\label{jacsquare:eqn14}
		\(\log \T_4\)'(\frac{\pi}{4}|\tau)=\frac{\VT_2^3(\tau)
			\VT_1'(\tau)}{4\T_4^4\(\frac{\pi}{4}\Big|\tau\)}=
		4q^{1/2}\prod_{n=1}^\infty (1-q^{2n})^2(1+q^{2n})^4,
	\end{equation}
	by using the infinite product representation for $\VT_1'(\tau)$ and $\VT_2(\tau)$ and
	\[
	\T_4\(\frac{\pi}{4}|\tau\)=\prod_{n=1}^\infty (1-q^n) (1+q^{2n-1}).
	\]
	Appealing to the trigonometric series expansion for the partial logarithmic derivative of $\T_4(z|\tau)$ in \reff{jabel:eqn38} we easily find that 
	\begin{equation}\label{jacsquare:eqn15}
		\(\log \T_1\)'(\frac{\pi}{4}|\tau)=4q^{1/2}\sum_{n=0}^\infty (-1)^n \frac{q^n}{1-q^{2n+1}}.
	\end{equation}
	Combining the above equation  and \reff{jacsquare:eqn14} and then using the second identity in \reff{jabel:eqn37}, we complete the proof of Theorem~\ref{2trigothm}.
\end{proof}
By choosing $F(z|\tau)=e^{2iz}\T_3(4z+\pi\tau|4\tau)$ in  Theorem~\ref{KLthm}  and making some elementary calculations we deduce that
\begin{equation}\label{KL:eqn14}
	e^{2iz}\T_3(4z+\pi\tau|4\tau)-e^{-2iz}\T_3(4z-\pi\tau|4\tau)=iq^{-1/8}\T_1(2z|\tau).
\end{equation}
%%%%%%%%%%%%%%%%%%%%%%%%%%%%%%%%%%%%%%%%%%%%%%%%%%%%%%%%%%%%%%%%%%%%%%%%%%%%%
\subsection{Parameterizations of two Eisenstein series identities related to modular equations of degree $5$ due to Ramanujan}
%%%%%%%%%%%%%%%%%%%%%%%%%%%%%%%%%%%%%%%%%%%%%%%%%%%%%%%%%%%%%%%%%%%%%%%%%%%%
\begin{thm}\label{RLthm}Let $\(\frac{a}{p}\)$ be the Legendre symbol modulo $p$ and let $\eta(\tau)$ be the Dedekind eta function. Then we have
	\begin{equation}\label{RLiu:eqn1}	
		\frac{\sin z  \sin 2 z}{\sin 5z}-\sum_{n=1}^\infty 
		\(\frac{n}{5}\) \frac{q^n}{1-q^n} \sin 2nz
		=\frac{ \eta^2(5\tau)\T_1(z|\tau)\T_1(2z|\tau)}{2\eta(\tau)\T_1(5z|5\tau)},	
	\end{equation}	
	and 
	\begin{align}\label{RLiu:eqn2}
		\sum_{n=1}^\infty \frac{(q^n-q^{2n}-q^{3n}+q^{4n})}{1-q^{5n}} \sin 2nz
		=\frac{\eta^2(\tau)\T_1(z|5\tau)\T_1(2z|5\tau)}{2\eta(5\tau)\T_1(z|\tau)}.
	\end{align}
\end{thm}
The identity in \reff{RLiu:eqn1} was first established by the author in  \cite[Proposition~5.1]{Liu2012JNT},  and \reff{RLiu:eqn2} is implied in
\cite[Theorem~1]{Liu2007JRMS} which was used to study the Eisenstein series associated with $\Gamma_0(5).$  For these two identities, see also \cite{Liu2021RamJ}.

Next we will show that the above two identities  are simple corollaries of Theorem~\ref{KLthm}.
\begin{proof} By a simple calculation we can verify that $F(z|\tau)$ satisfies the conditions of Theorem~\ref{KLthm}, where $F(z|\tau)$ is defined by
\[
F(z|\tau)=\frac{\T_1(5z|5\tau)}{\T_1(z|\tau)} 
\( 	\frac{\sin z~  \sin 2 z}{\sin 5z}-\sum_{n=1}^\infty 
\(\frac{n}{5}\) \frac{q^n}{1-q^n} \sin 2nz \).
\]	
Therefore there exists a constant $C$ independent of $z$ such that
\begin{equation}\label{RLiu:eqn3}
\frac{\T_1(5z|5\tau)}{\T_1(z|\tau)} 
\( 	\frac{\sin z~  \sin 2 z}{\sin 5z}-\sum_{n=1}^\infty 
\(\frac{n}{5}\) \frac{q^n}{1-q^n} \sin 2nz \)=C\T_1(2z|\tau).
\end{equation}	
Letting $z \to \pi/5$ in both sides of the above equation and using  L'Hospital's rule, we deduce that
\[
\VT_1'(0|5\tau) \(\sin \frac{\pi}{5} \sin \frac{2\pi}{5}\)
=C\T_1\(\frac{\pi}{5}|\tau\)\T_1\(\frac{2\pi}{5}|\tau\).  
\]
Substituting $4\sin \frac{\pi}{5} \sin \frac{2\pi}{5}=\sqrt{5}$	and 
$\T_1\(\frac{\pi}{5}|\tau\)\T_1\(\frac{2\pi}{5}|\tau\)=\sqrt{5} \eta(\tau)\eta(5\tau)$ into the above equation, we deduce that
\[
C=\frac{\eta^2(5\tau)}{2\eta(\tau)}.
\]
Substituting the above equation into \reff{RLiu:eqn3}, we arrive at \reff{RLiu:eqn1}.

Using the same method of proving \reff{RLiu:eqn1}, we can prove \reff{RLiu:eqn2} by using Theorem~\ref{KLthm}.
\end{proof}

\begin{conj}\label{liuconj} When $q \to 1$, the identity 
	in \reff{RLiu:eqn1} reduces to the formula
	\begin{equation}\label{conj:eqn1}
		\sum_{n=1}^\infty \(\frac{n}{5}\) \frac{x^n}{n}=\frac{1}{\sqrt{5}}
		\log \(\frac{2x^2+x+2+\sqrt{5}~ x}{2x^2+x+2-\sqrt{5}~x}\).
	\end{equation}
\end{conj}

Using the method similar to the proof of  Theorem~\ref{RLthm}, we can derive the following theorem.
\begin{thm}\label{RamCarlitz} We have 
\begin{equation}\label{ramcar:eqn1}
\sum_{n=1}^\infty \frac{nq^n}{1+q^n+q^{2n}} \sin 2nz
=\frac{\eta^3(\tau)\T_1(2z|3\tau)\T_1^2(z|3\tau)}{2\T_1^2(z|\tau)},	
\end{equation}
and
\begin{equation}\label{ramcar:eqn2}
	\frac{\sin^2 z \sin 2z}{\sin^2 3z}	-\sum_{n=1}^\infty \(\frac{n}{3}\) \frac{nq^n}{1-q^n} \sin 2nz
	=\frac{\eta^3(3\tau)\T_1^2(z|\tau)\T_1(2z|\tau)}{2\T_1^2(3z|3\tau)}.
\end{equation}			
\end{thm}

Dividing both sides of \reff{ramcar:eqn1} by $z$ and then letting $z\to 0,$ we obtain Ramanujan's identity \cite[p.~4212]{BBG1995}, \cite[Eq.(1.15)]{Liu1999Gainesville}
\begin{equation}\label{ramcar:eqn3}
	\sum_{n=1}^\infty \frac{n^2 q^n}{1+q^n+q^{2n}}=\frac{\eta^9(3\tau)}{\eta^3(\tau)}.
\end{equation}

Dividing both sides of \reff{ramcar:eqn2} by $z$ and then letting $z\to 0,$ we arrive at Carlitz's identity \cite[Eq.(3.1)]{Carlitz1953}
\begin{equation}\label{ramcar:eqn4}
	1-9\sum_{n=1}^\infty \(\frac{n}{3}\) \frac{n^2 q^n}{1-q^n}=
	\frac{\eta^9(\tau)}{\eta^3(3\tau)}.
\end{equation}

By setting $z=\pi/3$ in \reff{ramcar:eqn1} and noting that, 
$2\sin \frac{2n\pi}{3}=\sqrt{3} \(\frac{n}{3}\)$, we find that
\begin{equation}\label{ramcar:eqn5}
\sum_{n=1}^\infty \(\frac{n}{3}\)\frac{nq^n}{1+q^n+q^{2n}}=\frac{\eta^3(\tau)\eta^3(9\tau)}{\eta^2(3\tau)}.
\end{equation}

By taking $z=\pi/4$ in \reff{ramcar:eqn1} and noting that $\sin \frac{n \pi }{2}=\(\frac{-4}{n}\)$, we conclude that ( see also \cite[Eq.(5.24)]{ChenChen})
\begin{equation}\label{ramcar:eqn6}
	1-12\sum_{n=1}^\infty \(\frac{12}{n}\)\frac{nq^n}{1-q^n}=\frac{\eta(\tau)\eta(3\tau)\eta^2(4\tau)\eta^2(6\tau)}{\eta^2(12\tau)}.
	\end{equation}

%%%%%%%%%%%%%%%%%%%%%%%%%%%%%%%%%%%%%%%%%%%%%%%%%%%%%%%%%%%%%%%%%%%%%%
\subsection{The addition formula for the Weierstrass function}
If we specialize  Theorem~\ref{KLthm} to the case when
\begin{align*}
F(z|\tau)&=\T_1(z|\tau)\T_1(z-x|\tau)\T_1(z-y|\tau)\T_1(z+x+y|\tau)\\
&\quad -\T_1(z|\tau)\T_1(z+x|\tau)\T_1(z+y|\tau)\T_1(z-x-y|\tau),
\end{align*}
we conclude that
\begin{align} \label{Gauss:eqn1}
	&\T_1(z|\tau)\T_1(z-x|\tau)\T_1(z-y|\tau)\T_1(z+x+y|\tau)\\
	&-\T_1(z|\tau)\T_1(z+x|\tau)\T_1(z+y|\tau)\T_1(z-x-y|\tau)\nonumber\\
	&=\T_1(x|\tau)\T_1(y|\tau)\T_1(x+y|\tau) \frac{\T_1(2z|\tau)}{\T_1(z|\tau)}.\nonumber
\end{align}
Differentiating the above equation with respect to $z$, twice, and then setting $z=0$, we find the following proposition.
\begin{prop} \label{Gaussthm:n1}
	For $x\not \equiv 0 \pmod \Lambda, y\not \equiv 0 \pmod \Lambda$ and 
	$x+y\not \equiv 0 \pmod \Lambda$, we have	
	\begin{align}\label{Gauss:eqn2}
		&\((\log \T_1)'(x|\tau)+(\log \T_1)'(y|\tau)-(\log \T_1)'(x+y|\tau)\)^2\\
		&=\wp(x|\tau)+\wp(y|\tau)+\wp(x+y|\tau),\nonumber
	\end{align}
	or
	\begin{equation}\label{Gauss:eqn3}
		\{ \cot x +\cot y -\cot (x+y)+4\sum_{n=1}^\infty \frac{q^n}{1-q^n}\(\sin 2nx +\sin 2ny-\sin 2n (x+y)\) \}^2	
	\end{equation}
	\begin{align*}
		&=-L(\tau)+3+\cot^2 x +\cot^2 y +\cot^2 (x+y)\\
		&\qquad \qquad-8 \sum_{n=1}^\infty \frac{nq^n}{1-q^n}\(\cos 2nx +\cos 2ny+\cos 2n (x+y)\).
	\end{align*}		
\end{prop}
Next  we will use  Proposition~\ref{Gaussthm:n1} to give a proof of the addition formula for the Weierstrass elliptic function \cite[p.~34, Theorem~6]{Chandrasekharan1985}.
\begin{prop}\label{Gaussthm:n2}For $x\not \equiv 0 \pmod \Lambda, y\not \equiv 0 \pmod \Lambda$ and $x+y\not \equiv 0 \pmod \Lambda$, we have
	\begin{equation}\label{Gauss:eqn4}
		\wp(x|\tau)+\wp(y|\tau)+\wp(x+y|\tau)=\frac{1}{4}\(\frac{\wp'(x|\tau)-\wp'(y|\tau)}{\wp(x|\tau)-\wp(y|\tau)}\)^2.
	\end{equation}		
\end{prop}
\begin{proof} Replacing $x$ by $x+t$ and $y$ by $y+t$ in the addition formula for the  Weierstrass sigma function in \reff{jabel:eqn25} we find that
	\begin{equation*}
		\wp(x+t|\tau)-\wp(y+t|\tau)=-\VT_1'(\tau)^2 \frac{\T_1(x+y+2t|\tau)\T_1(x-y|\tau)}
		{\T_1^2(x+t|\tau)\T_1^2(y+t|\tau)}.
	\end{equation*}	
	Logarithmically differentiating the above equation with respect to $t$ and then putting $t=0$ gives 
	\[
	\frac{\wp'(x|\tau)-\wp'(y|\tau)}{\wp(x|\tau)-\wp(y|\tau)}
	=2(\log \T_1)'(x+y|\tau)-2(\log \T_1)'(x|\tau)-2(\log \T_1)'(y|\tau).
	\]
	Substituting the above equation into \reff{Gauss:eqn2}  we complete the proof of Proposition~\ref{Gaussthm:n2}.
\end{proof}
If we specialize \reff{Gauss:eqn3} to the case when $x=y=\pi/3$, we obtain the Ramanujan identity \cite[p.402, Entry 18.2.9]{ABLost}
\begin{equation}\label{Gauss:eqn5}
	a^2(\tau)=\(1+6\sum_{n=1}^\infty \(\frac{n}{3}\)\frac{q^n}{1-q^n}\)^2
	=1+12\sum_{n=1}^\infty \frac{nq^n}{1-q^n}-36\sum_{n=1}^\infty \frac{nq^{3n}}{1-q^{3n}}.	
\end{equation}

Setting $x=\frac{\pi}{7}$ and $y=\frac{2\pi}{7}$ in \reff{Gauss:eqn3}, we arrive at another identity of Ramanujan \cite[p.403, Entry 18.2.12]{ABLost}
\begin{equation}\label{Gauss:eqn6}
	\(1+2\sum_{n=1}^\infty \(\frac{n}{7}\)\frac{q^n}{1-q^n}\)^2	
	=1+4\sum_{n=1}^\infty \frac{nq^n}{1-q^n}-28\sum_{n=1}^\infty \frac{nq^{7n}}{1-q^{7n}}.
\end{equation}
%%%%%%%%%%%%%%%%%%%%%%%%%%%%%%%%%%%%%%%%%%%%%%%%%%%%%%%%%%%%%%%%%%%%%%%%
\subsection{Fourier series expansions for the quotients of theta functions}
%%%%%%%%%%%%%%%%%%%%%%%%%%%%%%%%%%%%%%%%%%%%%%%%%%%%%%%%%%%%%%%%%%%%%%%%%%%%%
Shen \cite{Shen1994TAMS} found several amazing  Fourier series expansions for  quotients of theta functions. Now we will use Theorem~\ref{KLthm} to recover the results due to Shen and derive similar new results. The following formula is equivalent to an identity in \cite[Eq.(1.6)]{Shen1994TAMS}.
\begin{thm}\label{JFourierthm:n1} We have
	\begin{equation}\label{JF:eqn1}
		(\log \T_1)'(x|2\tau)-(\log \T_4)'(x|2\tau)+(\log \T_1)'(y|2\tau)-(\log \T_4)'(y|2\tau)
	\end{equation}	
	\[
	=2\(\prod_{n=1}^\infty \frac{1-q^n}{1+q^n}\)^2 \frac{\T_1(x+y|2\tau)\T_4(x-y|2\tau)}{\T_1(x|\tau)\T_1(y|\tau)},
	\]	
	or
	\begin{equation}\label{JF:eqn2}
		\cot x+\cot y -4\sum_{n=1}^\infty \frac{q^n}{1+q^n} \(\sin 2nx +\sin 2ny\)	
	\end{equation}
	\begin{equation*}
		=2\(\prod_{n=1}^\infty \frac{1-q^n}{1+q^n}\)^2 \frac{\T_1(x+y|2\tau)\T_4(x-y|2\tau)}{\T_1(x|\tau)\T_1(y|\tau)}.	
	\end{equation*}
\end{thm}
\begin{proof}
	By Theorem~\ref{KLthm} we know that there exists a constant independent of $z$ such that 
	\begin{align*}
		&\T_1(z+x|\tau)\T_4(z-x|\tau)\T_1(z+y|\tau)\T_4(z-y|\tau)\\
		&-\T_1(z-x|\tau)\T_4(z+x|\tau)\T_1(z-y|\tau)\T_4(z+y|\tau)\\
		&=C\T_1(2z|\tau).
	\end{align*}
Setting $z=x$ in the above equation  we find that $C=\VT_4(\tau)\T_1(x+y|\tau)\T_4(x-y|\tau)$. It follows that
	\begin{align}\label{JF:eqn3}
		&\T_1(z+x|\tau)\T_4(z-x|\tau)\T_1(z+y|\tau)\T_4(z-y|\tau)\\
		&-\T_1(z-x|\tau)\T_4(z+x|\tau)\T_1(z-y|\tau)\T_4(z+y|\tau)\nonumber\\
		&=\VT_4(\tau) \T_1(2z|\tau)\T_1(x+y|\tau)\T_4(x-y|\tau).\nonumber
	\end{align}
	Differentiating through the above equation with respect to $z$ and then putting $z=0$, we deduce that
	\begin{equation*}
		(\log \T_1)'(x|\tau)-(\log \T_4)'(x|\tau)+(\log \T_1)'(y|\tau)-(\log \T_4)'(y|\tau)
	\end{equation*}	
	\[
	=\frac{\VT_1'(\tau)\VT_4(\tau)\T_1(x+y|\tau)\T_4(x-y|\tau)}
	{\T_1(x|\tau)\T_4(x|\tau)\T_1(y|\tau)\T_4(y|\tau)}.
	\]	
	Replacing $\tau$ by $2\tau$ in the above equation and then making use of 
	$2\T_1(z|2\tau)\T_4(z|2\tau)=\VT_2(\tau)\T_1(z|\tau)$ and $2\VT_1'(2\tau)\T_4(2\tau)=\VT_1'(\tau)\VT_2(\tau)$ in the resulting equation we arrive at \reff{JF:eqn1}.
	
	Substitutions of the trigonometric series expansions of the partial logarithmic derivatives of $\T_1$ and $\T_4$ into the left-hand side of 
	\reff{JF:eqn1} we get \reff{JF:eqn2}. We complete the proof of Theorem~\ref{JFourierthm:n1}.	
\end{proof}	
Taking $x=y$ in \reff{JF:eqn2} and employing $\T_1(x|\tau)\T_2(x|\tau)
=\VT_4(0|2\tau)\T_1(2x|2\tau)$ in the resulting equation we arrive at the Jacobi identity \cite[pp~.511--512]{WhiWat}.
\begin{equation}\label{JF:eqn4}
	\cot x-4\sum_{n=1}^\infty \frac{q^n}{1+q^n} \sin 2nx 	
	=\(\prod_{n=1}^\infty \frac{1-q^n}{1+q^n}\)^2 \frac{\T_2(x|\tau)}{\T_1(x|\tau)}.	
\end{equation}
\begin{thm}\label{JFourierthm:n2} We have
	\begin{align}\label{JF:eqn5}
		&1+2\sum_{n=1}^\infty \frac{q^{n/2}}{1+q^n} \(\cos 2nx+\cos 2ny\)\\
		&=\(\prod_{n=1}^\infty \frac{1-q^n}{1+q^n}\)^2
		\frac{\T_4(x+y|2\tau)\T_4(x-y|2\tau)}{\T_4(x|\tau)\T_4(y|\tau)},\nonumber
	\end{align}	
	and
	\begin{align}\label{JF:eqn6}
		&\sum_{n=1}^\infty \frac{q^{n/2}}{1+q^n} (\cos 2nx-\cos 2ny)\\
		&=-\(\prod_{n=1}^\infty \frac{1-q^n}{1+q^n}\)^2\frac{\T_1(x+y|2\tau)\T_1(x-y|
			2\tau)}{2\T_4(x|\tau)\T_4(y|\tau)}.\nonumber
	\end{align}
\end{thm}
These two series expansions are equivalent to the two Fourier series expansion of Shen in \cite[Eq.(1.11)]{Shen1994TAMS}.
\begin{proof}
	Employing the infinite product representations of $\T_1$ and $\T_4$ and a simple calculation, we find that
	\begin{equation}\label{JF:eqn7}
		\(\log \T_1\)'(z+\frac{\pi\tau}{2}|2\tau)-\(\log \T_4\)'(z+\frac{\pi\tau}{2}|2\tau)
		=-i-4i\sum_{n=1}^\infty \frac{q^{n/2}}{1+q^n} \sin 2nz.
	\end{equation}
	By replacing $(x, y)$ by $(x+\frac{\pi\tau}{2}, y+\frac{\pi\tau}{2})$ in \reff{JF:eqn1} and simplifying we find that
	\begin{align*}
		&(\log \T_1)'(x+\frac{\pi\tau}{2}|2\tau)-(\log \T_4)'(x+\frac{\pi\tau}{2}|2\tau)\\
		&+(\log \T_1)'(y+\frac{\pi\tau}{2}|2\tau)-(\log \T_4)'(y+\frac{\pi\tau}{2}|2\tau)\nonumber
	\end{align*}	
	\[
	=-2i\(\prod_{n=1}^\infty \frac{1-q^n}{1+q^n}\)^2 \frac{\T_4(x+y|2\tau)\T_4(x-y|2\tau)}{\T_4(x|\tau)\T_4(y|\tau)}.
	\]
	Substituting \reff{JF:eqn7}	 into the left-hand side of the above equation and canceling out the common factor $-2i$, we get \reff{JF:eqn5}.
	
	Replacing $y$ by $-y$ in \reff{JF:eqn1} and then replacing $(x, y)$ by $(x+\frac{\pi\tau}{2}, y+\frac{\pi\tau}{2})$ in the resulting equation and simplifying we find that
	\begin{align*}
		&(\log \T_1)'(x+\frac{\pi\tau}{2}|2\tau)-(\log \T_4)'(x+\frac{\pi\tau}{2}|2\tau)\\
		&-(\log \T_1)'(y+\frac{\pi\tau}{2}|2\tau)+(\log \T_4)'(y+\frac{\pi\tau}{2}|2\tau)
	\end{align*}	
	\[
	=2i\(\prod_{n=1}^\infty \frac{1-q^n}{1+q^n}\)^2 \frac{\T_1(x-y|2\tau)\T_1(x+y|2\tau)}{\T_4(x|\tau)\T_4(y|\tau)}.
	\]
	Substituting \reff{JF:eqn7}	 into the left-hand side of the above equation and simplifying gives \reff{JF:eqn6}. Hence we complete the proof of Theorem~
	\ref{JFourierthm:n2}. 
\end{proof}
When $y=x$ the identity in  \reff{JF:eqn5} becomes the Jacobi identity  \cite[pp~.511--512]{WhiWat}
\begin{align}\label{JF:eqn8}
	1+2\sum_{n=1}^\infty \frac{q^{n/2}}{1+q^n} \cos 2nx
	=\(\prod_{n=1}^\infty \frac{1-q^n}{1+q^n}\)^2
	\frac{\T_3(x|\tau)}{\T_4(x|\tau)}.
\end{align}

If we set $x=\pi/5$ and $y=2\pi/5$ in \reff{JF:eqn6} and then use  the finite trigonometric evaluation
\[
\cos \frac{2n\pi}{5}-\cos \frac{2n\pi}{5}=\frac{\sqrt{5}}{2}\(\frac{n}{5}\)
\]
in the resulting equation,  and finally replacing $q$ by $q^2$ we conclude that
\cite[Eq.(3.5)]{BerkYesi2009}
\begin{equation}\label{JF:eqn6a}
	\sum_{n=1}^\infty \(\frac{n}{5}\) \frac{q^n}{1+q^{2n}}=\frac{\eta(\tau)\eta(2\tau)\eta(10\tau)\eta(20\tau)}
	{\eta(4\tau)\eta(5\tau)}.
\end{equation}

\begin{thm}\label{JFourierthm:n3} We have
	\begin{equation}\label{JF:eqn9}
		(\log \T_1)'\(\frac{x}{2}|\tau\)-\(\log \T_2\)'\(\frac{x}{2}|\tau\)+(\log \T_1)'\(\frac{y}{2}|\tau\)-(\log \T_2)'\(\frac{y}{2}|\tau\)
	\end{equation}	
	\[
	=\frac{\VT_2^2(\tau)\T_1(\frac{x+y}{2}|\tau)\T_2(\frac{x-y}{2}|\tau)}
	{\T_1(x|2\tau)\T_2(y|2\tau)},
	\]
	or
	\begin{align}\label{JF:eqn10}
		&\csc x +\csc y +4\sum_{n=0}^\infty \frac{q^{2n+1}}{1-q^{2n+1}} \(\sin (2n+1)x +\sin (2n+1)y \)\\
		&=\frac{\T_2^2(0|\tau)\T_1(\frac{x+y}{2}|\tau)\T_2(\frac{x-y}{2}|\tau)}{2\T_1(x|2\tau)\T_1(y|2\tau)}.\nonumber
	\end{align}	
\end{thm}
\begin{proof}
By Theorem~\ref{KLthm} we know that there exists a constant $C$ independent of $z$ such that 
	\begin{align*}
		&\T_1(z+x|\tau)\T_2(z-x|\tau)\T_1(z+y|\tau)\T_2(z-y|\tau)\\
		&-\T_1(z-x|\tau)\T_2(z+x|\tau)\T_1(z-y|\tau)\T_2(z+y|\tau)\\
		&=C\T_1(2z|\tau),
	\end{align*}
	where $C$ independent of $z$. Setting $z=x$ in the above equation  we find that $C=\VT_2(\tau)\T_1(x+y|\tau)\T_2(x-y|\tau)$. It follows that
	\begin{align}\label{JF:eqn11}
		&\T_1(z+x|\tau)\T_2(z-x|\tau)\T_1(z+y|\tau)\T_2(z-y|\tau)\\
		&-\T_1(z-x|\tau)\T_2(z+x|\tau)\T_1(z-y|\tau)\T_2(z+y|\tau)\nonumber\\
		&=\VT_2(\tau) \T_1(2z|\tau)\T_1(x+y|\tau)\T_2(x-y|\tau).\nonumber
	\end{align}
	Differentiating through the above equation with respect to $z$ and then putting $z=0$, we deduce that
	\begin{equation*}
		(\log \T_1)'(x|\tau)-(\log \T_2)'(x|\tau)+(\log \T_1)'(y|\tau)-(\log \T_2)'(y|\tau)
	\end{equation*}	
	\[
	=\frac{\VT_1'(\tau)\VT_2(\tau)\T_1(x+y|\tau)\T_2(x-y|\tau)}
	{\T_1(x|\tau)\T_2(x|\tau)\T_1(y|\tau)\T_2(y|\tau)}.
	\]
	Using $\T_1(z|\tau)\T_2(z|\tau)=\VT_4(2\tau)\T_1(2z|2\tau)$ and $\VT_4^2(2\tau)=\VT_3(\tau)\VT_4(\tau)$  in the right-hand side of the above equation and finally replacing $(x, y)$ by $(x/2, y/2)$ we obtain \reff{JF:eqn9}.
	
	Substitutions of the trigonometric series expansions of the partial logarithmic derivatives of $\T_1$ and $\T_2$ into the left-hand side of 
	\reff{JF:eqn9} we get \reff{JF:eqn10}. We complete the proof of Theorem~\ref{JFourierthm:n3}.	
\end{proof}
After replacing $q$ by $-q$ in \reff{JF:eqn10}  we  arrive at the Fourier series expansion in \cite[Eq.(2.4)]{Shen1994TAMS}.  
If we set $y=x$ in \reff{JF:eqn10}, we arrive at the Jacobi identity \cite[pp~.511--512]{WhiWat}.
\begin{equation}\label{JF:eqn10a}
	\csc x+4\sum_{n=0}^\infty \frac{q^{2n+1}}{1-q^{2n+1}} \sin (2n+1)x=
	\VT_2^2(\tau) \frac{\T_4(x|2\tau)}{2\T_1(x|2\tau)}.
\end{equation}
Using Theorem~\ref{KLthm} and some simple calculation we can obtain the following theorem.
\begin{thm}\label{JFourierthm:n4} We have
	\begin{equation}\label{JF:eqn12}
		\(\log \T_4\)'(x|\tau)-	\(\log \T_3\)'(x|\tau)+\(\log \T_4\)'(y|\tau)-	\(\log \T_3\)'(y|\tau)
	\end{equation}	
	\[
	=\VT_2^2(\tau) \frac{\T_1(x+y|\tau)\T_2(x-y|\tau)}{\T_4(2x|2\tau)\T_4(2y|2\tau)},
	\]
	or
	\begin{equation}\label{JF:eqn13}
		\sum_{n=0}^\infty \frac{q^{n+1/2}}{1-q^{2n+1}} \(\sin 2(2n+1)x+\sin 2(2n+1)y\)
	\end{equation}
	\[
	=\VT_2^2(\tau) \frac{\T_1(x+y|\tau)\T_2(x-y|\tau)}{8\T_4(2x|2\tau)\T_4(2y|2\tau)}.
	\]
\end{thm}
If we set $y=x$ in \reff{JF:eqn13} and then  replace $2x$ by $x$, we arrive at the Jacobi identity \cite[pp.511--512]{WhiWat}
\begin{equation}
	8\sum_{n=0}^\infty \frac{q^{n+1/2}}{1-q^{2n+1}} \sin (2n+1)x=\VT_2^2(\tau)\frac{\T_1(x|2\tau)}{\T_4(x|2\tau)}.
\end{equation}

Using Theorem~\ref{KLthm} we can  prove the  following theta function identity:
\begin{align}\label{JF:eqn16}
	&\T_1\(z+x+\frac{\pi}{4}|\tau\)\T_1\(z-x+\frac{\pi}{4}|\tau\)\T_1\(z+y+\frac{\pi}{4}|\tau\)\T_1\(z-y+\frac{\pi}{4}|\tau\)\\
	&-\T_1\(z-x-\frac{\pi}{4}|\tau\)\T_1\(z+x-\frac{\pi}{4}|\tau\)\T_1\(z-y-\frac{\pi}{4}|\tau\)\T_1\(z+y-\frac{\pi}{4}|\tau\)\nonumber\\
	&=\VT_2(\tau)\T_2(x+y|\tau)\T_2(x-y|\tau)\T_1(2z|\tau).\nonumber
\end{align}
Differentiating through the above equation with respect to $z$ and then setting $z=0$ and simplifying we obtain the following new Fourier series expansion.
\begin{thm}\label{JFourierthm:n6}  We have
	\begin{align*}
		&\frac{1}{\cos 2x}+\frac{1}{\cos 2y}+4\sum_{n=1}^\infty \(\frac{-4}{n}\)
		\frac{q^n}{1-q^n}\(\cos 2nx +\cos 2ny\)\\
		&=\VT_2^2(\tau)\frac{\T_2(x+y|\tau)\T_2(x-y|\tau)}{2\T_2(2x|2\tau)\T_2(2y|\tau)}.
	\end{align*}
\end{thm}
The following theorem can also be derived easily from Theorem~\ref{KLthm}. 
\begin{thm}\label{KJRthm} Suppose that $s+t+u+v$ is an integral multiple of $\pi$. Then we have
	\begin{align}\label{KR:eqn15}
		&\T_2(z-s|\tau)\T_2(z-t|\tau)\T_2(z-u|\tau)\T_2(z-v|\tau)\\
		&-\T_2(z+s|\tau)\T_2(z+t|\tau)\T_2(z+u|\tau)\T_2(z+v|\tau)\nonumber\\
		&=\T_1(s+t|\tau)\T_1(s+u|\tau)\T_1(s+v|\tau)\T_1(2x|\tau),\nonumber
	\end{align}	
	and 
	\begin{equation}\label{KR:eqn16}
		\(\log \T_2\)'(s|\tau)+\(\log \T_2\)'(t|\tau)+\(\log \T_2\)'(u|\tau)+\(\log \T_2\)'(v|\tau)
	\end{equation}
	\begin{equation*}
		=-\frac{\VT_1'(\tau)\T_1(s+t|\tau)\T_1(s+u|\tau)\T_1(s+v|\tau)}{\T_2(s|\tau)\T_2(t|\tau)\T_2(u|\tau)\T_2(v|\tau)}.
	\end{equation*}	
\end{thm}
By choosing $s=x,~t=y, ~u=-x-y$ and $v=0$ in Theorem~\ref{KJRthm} and simplifying we obtain the following proposition~\cite[Corollary~3.1]{Liu2010Mysore}.
\begin{prop} \label{JRthm} We have
	\begin{align}\label{KR:eqn17}
		&\tan x+\tan y-\tan (x+y)-4\sum_{n=1}^\infty \frac{(-q)^n}{1-q^n}\( \sin 2n x+ \sin 2ny-\sin 2n(x+y)\)\\
		&=-\(\prod_{n=1}^\infty \frac{1-q^n}{1+q^n}\)\frac{\T_1(x|\tau)\T_1(y|\tau)\T_1(x+y|\tau)}{\T_2(x|\tau)\T_2(y|\tau)\T_2(x+y|\tau)}.\nonumber
	\end{align}		
\end{prop}
Setting $x=\pi/7$ and $y=2\pi/7$ in \reff{KR:eqn17} and simplifying and finally replacing $q$ by $-q$,  we arrive at the Ramanujan identity \cite[p.304]{Berndt1991}, \cite[p. 403, Entry 18.2.12]{ABLost}
\begin{equation}\label{KR:eqn18}
	\phi(q)\phi(q^7)=1+2\sum_{n=1}^\infty \(\frac{n}{7}\) \frac{q^n}{1-(-q)^n}.
\end{equation}	
Setting $t=v=u$ and $s=-3u$ in \reff{KR:eqn16} and then 
substituting the trigonometric series expansion of the partial derivative of $\T_2(z|\tau)$ with respect to $z$ into the resulting equation,  we obtain the following proposition.
\begin{prop}\label{JRthm:n1} We have
	\begin{equation}\label{KR:eqn19}
		\tan 3u-3 \tan u +4\sum_{n=1}^\infty \frac{(-q)^n}{1-q^n} \(3\sin 2nu-\sin 6nu\)=\frac{\VT_1'(\tau)\T_1^3(2u|\tau)}
		{\T_2(3u|\tau)\T_2^3(u|\tau)}.
	\end{equation}
\end{prop}
Dividing both sides of the above equation by $u^3$ and
then letting $u\to 0$ and replacing $q$ by $-q$ we arrive at
Jacobi's formula for sums of eight squares (see, for example \cite[p.~70]{Berndt06} and \cite[Eq.(15.9)]{Roy2017})
\begin{equation}\label{KR:eqn20}
	\phi^8(q)=1+16\sum_{n=1}^\infty \frac{n^3 q^n}{1-(-q)^n}.
\end{equation}

Setting $u=\pi/3$ in \reff{KR:eqn19} and simplifying and finally  replacing $q$ by $-q$,  we immediately obtain that  \cite[p~.141, Eq.(iii)]{Berndt06}
\begin{equation}\label{KR:eqn21}
	1-2\sum_{n=1}^\infty \(\frac{n}{3}\) \frac{q^n}{1-(-q)^n}=\frac{\phi^3(q^3)}{\phi(q)}.
\end{equation}

With the help of Theorem~\ref{KLthm} we can prove the following general Lambert series identity.

\begin{thm} \label{addKLthm} If $s+t+u+v$ is an integral multiple of $\pi$, then we have 
	\begin{equation}\label{addKR:eqn1}
		\(\log \T_4\)'(s|\tau)+\(\log \T_4\)'(t|\tau)+\(\log \T_4\)'(u|\tau)+\(\log \T_4\)'(v|\tau)
	\end{equation}
	\begin{equation*}
		=-\frac{\VT_1'(\tau)\T_1(s+t|\tau)\T_1(s+u|\tau)\T_1(s+v|\tau)}{\T_4(s|\tau)\T_4(t|\tau)\T_4(u|\tau)\T_4(v|\tau)},
	\end{equation*}	
	or
	\begin{equation}\label{addKR:eqn2}
		4\sum_{n=1}^\infty \frac{q^{n/2}}{1-q^n} \(\sin 2n s + \sin 2nt +\sin 2n u+\sin 2nv\)
	\end{equation}
	\begin{equation*}
		=-\frac{\VT_1'(\tau)\T_1(s+t|\tau)\T_1(s+u|\tau)\T_1(s+v|\tau)}{\T_4(s|\tau)\T_4(t|\tau)\T_4(u|\tau)\T_4(v|\tau)}.
	\end{equation*}
\end{thm}
By choosing $s=x,~t=y, ~u=-x-y$ and $v=0$ in Theorem~\ref{addKLthm} and simplifying we obtain the following proposition~\cite[Eq.(3.11)]{Liu2010pac}:
\begin{align}\label{addKR:eqn3}
	&4\sum_{n=1}^\infty \frac{q^{n/2}}{1-q^n} \(\sin 2nx + \sin 2ny -\sin 2n (x+y)\)\\
	&=\frac{\VT_1'(\tau)\T_1(x|\tau)\T_1(y|\tau)\T_1(x+y|\tau)}
	{\VT_4(\tau)\T_4(x|\tau)\T_4(y|\tau)\T_4(x+y|\tau)}.\nonumber
\end{align}
Setting $x=\pi/7$ and $y=2\pi/7$ in the above equation and then replacing $q$ by $q^2$ we arrive at the Ramanujan identity \cite[p.~404, Entry 18.2.13]{ABLost}
\begin{equation}\label{addKR:eqn4}
	q\psi(q)\psi(q^7)=\sum_{n=1}^\infty \(\frac{n}{7}\) \frac{q^n}{1-q^{2n}}.
\end{equation}

Putting $y=x$ in \reff{addKR:eqn3} and then dividing both sides of the resulting equation by $x^3$ and letting $x \to 0$ and finally replacing $q$ by $q^2$, we find that \cite[Theorem~9]{Liu2003RamJ}
\begin{equation}\label{addKR:eqn5}
	q\psi^8(q)=\sum_{n=1}^\infty \frac{n^3 q^n}{1-q^{2n}}.	
\end{equation}
This formula is due to Legendre \cite{Legendre1828}, and Ramanujan \cite[p.144]{Raman1927} stated it without proof. 

\begin{prop} \label{lambertJacobi} We have
\begin{equation}\label{LJ:eqn1}
	1-4(\tan z)\sum_{n=1}^\infty \frac{q^n}{1+q^n} \sin 2nz 
	=\prod_{n=1}^\infty \frac{ (1+2q^n \cos 2z+q^{2n})(1-q^n)^2}{ (1-2q^n \cos 2z +q^{2n})(1+q^n)^2}.
\end{equation}	
\end{prop}	
\begin{proof}Using Theorem~\ref{KLthm} we can easily deduce that
\begin{align*}
	&\T_1(z+x|\tau)\T_2(z|\tau)\T_1(2z-x|2\tau)-\T_1(z-x|\tau)\T_2(z|\tau)\T_1(2z+x|2\tau)\\
	&=\T_2(x|\tau)\T_1(x|2\tau)\T_1(2z|\tau). 
\end{align*}		
Differentiating the above equation with respect to $z$ and then setting $z=0$ in the resulting equation, we find that
\[
2\(\log \T_1\)'(x|2\tau)-\(\log \T_1\)'(x|\tau)=\frac{\VT_1'(\tau)\T_2(x|\tau)}{\VT_2(\tau)\T_1(x|\tau)}.
\]
Substituting the trigonometric series expansion of $\(\log \T_1\)'(x|\tau)$ 
into the left-hand side of the above equation and applying the infinite product representations to the right-hand side of the above equation, we complete the proof of Proposition~\ref{lambertJacobi}.	
\end{proof}
Setting $z=\pi/6$ in \reff{LJ:eqn1} and then writing $q$ and $-q$ we conclude that  \cite[p.~141, Eq.(ii)]{Berndt06}, \cite[Eq.(3.10)]{Shen1994PAMS}
\[
1+2\sum_{n=1}^\infty \(\frac{n}{3}\) \frac{q^n}{1+(-q)^n}=\phi(q)\phi(q^3).
\]
If we set $z=\pi/4$ in \reff{LJ:eqn1} and then replace $q$ by $-q$, we obtain Jacobi's two-squares identity
\[
\phi^2(q)=1+4\sum_{n=0}^\infty (-1)^n \frac{q^{2n+1}}{1-q^{2n+1}}.
\]
Letting $x\to \pi/2$ in \reff{LJ:eqn1} and then replacing $q$ by $-q$, we obtain Jacobi's four-squares identity (see \cite[pp.59-61]{Berndt06}, \cite[Eq.(15.4)]{Roy2017})
\[
\phi^4(q)=1+8\sum_{n=1}^\infty \frac{nq^n}{1+(-q)^n}.
\]

Using Theorem~\ref{KLthm} we can also prove the following general Fourier series expansion for  the quotients of theta functions.
\begin{thm} \label{JFourierthm:n7}	
	Let $r_1, r_2, r_3,  r_4, s_1, s_2, s_3, s_4$ are rational numbers such that $r=r_1+r_2+r_3+r_4$ and $s=s_1+s_2+s_3+s_4$ are integers and let $u_1, u_2, u_3, u_4$ are complex numbers 
	such that $u_1+u_2+u_3+u_4=0.$ Then we have 
	\begin{align}
		&2ir+\sum_{k=1}^4 \(\log \T_1\)'(u_k+r_k \pi\tau+s_k \pi|\tau)\\
		&=\frac{e^{2ir(u_1+r_1\pi\tau+s_1\pi)}\VT_1'(\tau)}{\T_1(u_1+r_1 \pi\tau+s_1 \pi|\tau)}\prod_{k=2}^4 \frac{\T_1(u_1+u_k+(r_1+r_k)\pi\tau+(s_1+s_k)\pi|\tau)}{\T_4(u_k+r_k \pi\tau+s_k \pi|\tau)}.\nonumber
	\end{align}		
\end{thm}
%%%%%%%%%%%%%%%%%%%%%%%%%%%%%%%%%%%%%%%%%%%%%%%%%%%%%%%%%%%%%%%%%%%%
\section{An addition formula for the theta functions of degree $6$ and the Rogers--Ramanujan continued fraction}
%%%%%%%%%%%%%%%%%%%%%%%%%%%%%%%%%%%%%%%%%%%%%%%%%%%%%%%%%%%%%%%%%%%%%%%%%%%%%
We begin this section by proving Theorem~\ref{6RRCthm:n1} with the help of 
Theorem~\ref{liuaddthm}.
\begin{proof}  Let $F(z|\tau)$ and $G(z|\tau)$ be the given functions in Theorem~\ref{6RRCthm:n1}. For the time being we set
	\begin{equation}\label{rrc:eqn1}
		A(z|\tau)=F(z|\tau)-F(-z|\tau)\quad \text{and}\quad 
		B(z|\tau)=G(z|\tau)-G(-z|\tau).
	\end{equation}	
	Using Proposition~\ref{doubleperiods} we can verify that the even entire function of $z$ defined by 
	\begin{equation}\label{rrc:eqn2}
		f(z|\tau)=\(\frac{A(z|\tau)B(y|\tau)-B(z|\tau)A(y|\tau)}{\T_1(z-y|\tau)\T_1(z+y|\tau)}\)\T_1(2z|\tau),	
	\end{equation}
	satisfies the conditions of Theorem~\ref{liuaddthm}. Since $0, \pi/2, (\pi+\pi\tau)/2$ and $(\pi\tau)/2$ are zeros of $\T_1(2z|\tau)$, we immediately deduce that
	\begin{equation}\label{rrc:eqn3}
		f(0|\tau)=f(\pi/2|\tau)=f((\pi\tau)/2|\tau)=f((\pi+\pi\tau)/2|\tau)=0.
	\end{equation}
	By L'Hospital's rule  we easily find that
	\begin{equation}\label{rrc:eqn4}
		f(y|\tau)=\frac{A'(y|\tau)B(y|\tau)-B'(y|\tau)A(y|\tau)}{\VT_1'(\tau)}.
	\end{equation}
	Substituting the above values of $f$ into \reff{jabel:eqn23} in Theorem~\ref{liuaddthm}, we conclude that
	\begin{align}\label{rrc:eqn5}
		&\frac{A(x|\tau)B(y|\tau)-A(y|\tau)B(x|\tau)}{\T_1(x-y|\tau)\T_1(x+y|\tau)}\\
		&=\frac{\(A'(y|\tau)B(y|\tau)-B'(y|\tau)A(y|\tau)\)\T_1(2x|\tau)}{\VT_1'(\tau)\T_1^2(2y|\tau)}.\nonumber
	\end{align}
	The left-hand side of the above equation is symmetric about $x$ and $y$, so the right-hand side is also symmetric about $x$ and $y$. It follows that
	\begin{align*}
		&\frac{A'(y|\tau)B(y|\tau)-B'(y|\tau)A(y|\tau)}{\VT_1'(\tau)\T_1^3(2y|\tau)}\\
		&=\frac{A'(x|\tau)B(x|\tau)-B'(x|\tau)A(x|\tau)}{\VT_1'(\tau)\T_1^3(2x|\tau)}.
	\end{align*}
	From this equation we know that there exists a constant $C$ independent of $y$ such that
	\begin{equation}\label{rrc:eqn6}
		A'(y|\tau)B(y|\tau)-B'(y|\tau)A(y|\tau)=C\VT_1'(\tau)\T_1^3(2y|\tau).
	\end{equation}
	Substituting  the above equation into \reff{rrc:eqn5} and combining the resulting equation with \reff{rrc:eqn4} we complete the proof of Theorem~\ref{6RRCthm:n1}.
\end{proof}
The well-known Rogers--Ramanujan continued fraction is defined by
\begin{equation}\label{rrc:eqn7}
	R(\tau)=\cfrac{q^{1/5}}{1+\cfrac{q}{1+\cfrac{q^2}{1+\cfrac{q^3}{1+\cdots}}}}
\end{equation}
L. Rogers \cite{Rogers1894} used the Rogers-Ramanujan identities to
give the infinite product representation of $R(\tau)$ as follows:
\begin{equation}\label{rrc:eqn8}
	R(\tau)=q^{1/5}\prod_{n=1}^\infty \frac{(1-q^{5n-1})(1-q^{5n-4})}{(1-q^{5n-2})(1-q^{5n-3})}
	=e^{-\frac{3\pi i\tau}{5}}\frac{\T_1(\pi\tau|5\tau)}{\T_1(2\pi\tau|5\tau)}.
\end{equation}
 The following theorem first appeared in \cite[Propostion~4.2]{Liu2012JNT} without proof. Now we use Theorem~\ref{6RRCthm:n1} to  prove it. 
\begin{thm}\label{6RRCthm:n2} We have
	\begin{align}\label{rrc:eqn9}
		&\(\T_1\(x+\frac{\pi}{5}|\tau \)-\T_1\(x-\frac{\pi}{5}|\tau \)\)\(\T_1\(y+\frac{2\pi}{5}|\tau \)-\T_1\(y-\frac{2\pi}{5}|\tau \)\)\\
		&-\(\T_1\(y+\frac{\pi}{5}|\tau \)-\T_1\(y-\frac{\pi}{5}|\tau \)\)\(\T_1\(x+\frac{2\pi}{5}|\tau \)-\T_1\(x-\frac{2\pi}{5}|\tau \)\)\nonumber\\
		&=\frac{5\T_1(x+y|5\tau)\T_1(x-y|5\tau)\T_1(2x|5\tau)\T_1(2y|5\tau)}
		{\T_1(x|5\tau)\T_1(y|5\tau)}.\nonumber
	\end{align}		
\end{thm}
\begin{proof} Using the functional equations for $\T_1$ in Proposition~\ref{doubleperiods} we can easily find that for any 
	integer $k$,	
	\begin{align}\label{rrc:eqn10}
		&\T_1\(\frac{k\pi+\pi\tau}{5}\Big|\frac{\tau}{5}\)=-q^{-\frac{1}{10}}e^{-\frac{2k\pi i}{5}}
		\T_1\(\frac{k\pi}{5}\Big|\frac{\tau}{5}\),\\
		&\T_1\(\frac{k\pi+2\pi\tau}{5}\Big|\frac{\tau}{5}\)=q^{-\frac{2}{5}}e^{-\frac{4k\pi i}{5}}
		\T_1\(\frac{k\pi}{5}\Big|\frac{\tau}{5}\).\nonumber
	\end{align}		
	Taking $F(z|\tau)=\T_1(z|\tau)\T_1(z+\frac{\tau}{5}|\frac{\tau}{5})$ and 
	$G(z|\tau)=\T_1(z|\tau)\T_1(z+\frac{2\tau}{5}|\frac{\tau}{5})$	in  Theorem~\ref{6RRCthm:n1}, then we find for some constant $C$ independent of $x$ and $y$ that
	\begin{align}\label{rrc:eqn11}
		&\(\T_1\(x+\frac{\pi}{5}|\frac{\tau}{5} \)-\T_1\(x-\frac{\pi}{5}|\frac{\tau}{5} \)\)\(\T_1\(y+\frac{2\pi}{5}|\frac{\tau}{5} \)-\T_1\(y-\frac{2\pi}{5}|\frac{\tau}{5} \)\)\\
		&-\(\T_1\(y+\frac{\pi}{5}|\frac{\tau}{5} \)-\T_1\(y-\frac{\pi}{5}|\frac{\tau}{5} \)\)\(\T_1\(x+\frac{2\pi}{5}|\frac{\tau}{5} \)-\T_1\(x-\frac{2\pi}{5}|\frac{\tau}{5} \)\)\nonumber\\
		&=C\frac{\T_1(x+y|\tau)\T_1(x-y|\tau)\T_1(2x|\tau)\T_1(2y|\tau)}
		{\T_1(x|\tau)\T_1(y|\tau)}.\nonumber
	\end{align}
	Setting $x=\frac{2\pi \tau}{5}$ and $y=\frac{\pi \tau}{5}$ in the above equation and using \reff{rrc:eqn10} in the resulting equation and simplifying we find that
	\begin{align}\label{rrc:eqn12}
		&4\(\cos^2 \frac{2\pi}{5}-\cos^2 \frac{\pi}{5}\)q^{-\frac{1}{2}}~
		\T_1\(\frac{\pi}{5}|\frac{\tau}{5}\)\T_1\(\frac{2\pi}{5}|\frac{\tau}{5}\)\\
		&=Cq^{-\frac{2}{5}}~\T_1\(\frac{\pi\tau}{5}|\tau\)\T_1\(\frac{2\pi\tau}{5}|\tau\).\nonumber
	\end{align}
	Substituting $4\(\cos^2 \frac{\pi}{5}-\cos^2 \frac{2\pi}{5}\)=\sqrt{5}$ and 
	$\T_1\(\frac{\pi}{5}|\frac{\tau}{5}\)\T_1\(\frac{2\pi}{5}|\frac{\tau}{5}\)
	=\sqrt{5}\eta(\frac{\tau}{5})\eta(\tau)$ and 
	\[
	\T_1\(\frac{\pi\tau}{5}|\tau\)\T_1\(\frac{2\pi\tau}{5}|\tau\)
	=-q^{-\frac{1}{10}}\eta\(\frac{\tau}{5}\)\eta(\tau)
	\]
	into \reff{rrc:eqn11} we get $C=5$. Substituting this into \reff{rrc:eqn14} and then replacing $\tau$ by $5\tau$ we complete the proof of Theorem~\ref{6RRCthm:n2}.
\end{proof}	
By putting $x=2\pi/5$ and $y=\pi/5$ in Theorem~\ref{6RRCthm:n2} and simplifying we easily find the following proposition \cite[Eq.(1.9)]{LiuIntegers20010}. 
\begin{prop}\label{6RRCthm:n3} We have 
	\begin{equation}\label{rrc:eqn13}
		\frac{\T_1(\frac{2\pi}{5}|\tau)}{\T_1(\frac{\pi}{5}|\tau)}-\frac{\T_1(\frac{\pi}{5}|\tau)}{\T_1(\frac{2\pi}{5}|\tau)}=1+5\frac{\eta(25\tau)}{\eta(\tau)}.
	\end{equation}		
\end{prop}
Applying the first imaginary transformation formula in \reff{jabel:eqn19} (see also \cite[p.~177, Eq.(79.7)]{Rademacher1973})  to Theorem~\ref{6RRCthm:n2} we arrive at the following theorem \cite[Proposition~4.1]{Liu2012JNT}. 
\begin{thm}\label{6RRCthm:n4} Let $H_k(z|\tau)=e^{2ikz}\T_1(5z+k\pi\tau|5\tau)-e^{-2ikz}\T_1(5z-k\pi\tau|5\tau)$. Then we have
	\begin{align}\label{rrc:eqn14}
		&q^{1/2} H_1(x|\tau)H_1(y|\tau)-q^{1/2}H_1(y|\tau)H_2(x|\tau)\\
		&=\frac{\T_1(x+y|\tau)\T_1(x-y|\tau)\T_1(2x|\tau)\T_1(2y|\tau)}
		{\T_1(x|\tau)\T_1(y|\tau)}.	\nonumber	
	\end{align}		
\end{thm}
If we specialize Theorem~\ref{6RRCthm:n4} to the case when  $y=0$ and $x=\pi\tau$ in Theorem~\ref{6RRCthm:n4},  we conclude that
\[
q^{\frac{7}{10}}\T_1^2(2\pi\tau|5\tau)-q^{\frac{1}{10}}\T_1^2(\pi\tau|5\tau)
-q^{\frac{2}{5}}\T_1(\pi\tau|5\tau)\T_1(2\pi\tau|5\tau)
=\T_1\(\frac{\pi\tau}{5}\Big|\tau\)\T_1\(\frac{2\pi\tau}{5}\Big|\tau\).
\]
Dividing both sides of the above equation by $q^{\frac{2}{5}}\T_1(\pi\tau|5\tau)\T_1(2\pi\tau|5\tau)$ we deduce that
\begin{equation*}
	q^{\frac{3}{10}}\frac{\T_1(2\pi\tau|5\tau)}{\T_1(\pi\tau|5\tau)}
	-q^{-\frac{3}{10}}\frac{\T_1(\pi\tau|5\tau)}{\T_1(2\pi\tau|5\tau)}	
	=1+q^{-\frac{2}{5}}\frac{\T_1(\frac{\pi\tau}{5}|\tau)\T_1(\frac{2\pi\tau}{5}|\tau)}
	{\T_1(\pi\tau|5\tau)\T_1(2\pi\tau|5\tau)}.
\end{equation*}
Applying the infinite product representation for theta function $\T_1$, we find that the above equation is equivalent to the Ramanujan identity
(see also \cite[p.~1478]{LiuYang2009})
\begin{equation}\label{rrc:eqn15}
	R^{-1}(\tau)-R(\tau)=1+\frac{\eta(\frac{\tau}{5})}{\eta(5\tau)}.
\end{equation}
Using the same method as that of  proving Theorem~\ref{6RRCthm:n2}, by choosing $F(z|\tau)=\T_1(z|\tau)\T_1^5(z+\frac{\pi}{5}|\tau)$ and 
$G(z|\tau)=\T_1(z|\tau)\T_1^5(z+\frac{2\pi}{5}|\tau)$ in Theorem~\ref{6RRCthm:n1}, we can prove the following theorem.
\begin{thm}\label{6RRCthm:n5} We have
	\begin{align}\label{rrc:eqn16}
		&\(\T_1^5\(x+\frac{\pi}{5}\Big|\tau\)-\T_1^5\(x-\frac{\pi}{5}\Big|\tau\)\) \(\T_1^5\(y+\frac{2\pi}{5}\Big|\tau\)-\T_1^5\(y-\frac{2\pi}{5}\Big|\tau\)\)\\
		&~-\(\T_1^5\(y+\frac{\pi}{5}\Big|\tau\)-\T_1^5\(y-\frac{\pi}{5}\Big|\tau\)\)\(\T_1^5\(x+\frac{2\pi}{5}\Big|\tau\)-\T_1^5\(x-\frac{2\pi}{5}\Big|\tau\)\)\nonumber
	\end{align}
	\begin{equation*}
		=\(250\eta^4(\tau)\eta^4(5\tau)+3125\frac{\eta^{10}(5\tau)}{\eta^2(\tau)}\)
		\frac{\T_1(x+y|\tau)\T_1(x-y|\tau)\T_1(2x|\tau)\T_1(2y|\tau)}{\T_1(x|\tau)\T_1(y|\tau)}.\
	\end{equation*}
\end{thm}
Using Theorem~\ref{6RRCthm:n5} we can prove the following curious identity \cite[Eq.(1.12)]{LiuIntegers20010}.
\begin{prop}\label{6RRCthm:n6} We have 
	\begin{equation}\label{rrc:eqn17}
		\frac{\T_1^5(\frac{2\pi}{5}|\tau)}{\T_1^5(\frac{\pi}{5}|\tau)}
		-\frac{\T_1^5(\frac{\pi}{5}|\tau)} {\T_1^5(\frac{2\pi}{5}|\tau)}
		=11+125\frac{\eta^6(5\tau)}{\eta^6(\tau)}.
	\end{equation}		
\end{prop}
\begin{proof} Taking $x=2\pi/5$ and $y=\pi/5$ in Theorem~\ref{6RRCthm:n5} and simplifying we find that
	\begin{align*}
		&\T_1^{10}\(\frac{2\pi}{5}\Big|\tau\)-\T_1^{10}\(\frac{\pi}{5}\Big|\tau\)-\T_1^{5}\(\frac{\pi}{5}\Big|\tau\)\T_1^{5}\(\frac{2\pi}{5}\Big|\tau\)\\
		&=\(250\eta^4(\tau)\eta^4(5\tau)+3125\frac{\eta^{10}(5\tau)}{\eta^2(\tau)}\)
		\T_1\(\frac{\pi}{5}\Big|\tau\)\T_1\(\frac{2\pi}{5}\Big|\tau\).
	\end{align*}	
	Dividing both sides of the above equation by $\T_1^{5}\(\frac{\pi}{5}|\tau\)\T_1^{5}\(\frac{2\pi}{5}|\tau\)$	we deduce that
	\[
	\frac{\T_1^5(\frac{2\pi}{5}|\tau)}{\T_1^5(\frac{\pi}{5}|\tau)}
	-\frac{\T_1^5(\frac{\pi}{5}|\tau)} {\T_1^5(\frac{2\pi}{5}|\tau)}
	=1+\frac{250\eta^4(\tau)\eta^4(5\tau)+3125\frac{\eta^{10}(5\tau)}{\eta^2(\tau)}}{\T_1^{4}\(\frac{\pi}{5}|\tau\)\T_1^{4}\(\frac{2\pi}{5}|\tau\)}.
	\]
	Substituting $\T_1\(\frac{\pi}{5}|\tau\)\T_1\(\frac{2\pi}{5}|\tau\)=\sqrt{5}\eta(\tau)\eta(5\tau)$ into the left-hand side of the above equation, we complete  the proof of Proposition~\ref{6RRCthm:n6}.
\end{proof}	

Applying the imaginary transformation formula in the first equation in Proposition~\ref{imaginarypp} to Theorem~\ref{6RRCthm:n5} we are led to the following theorem. 
\begin{thm}\label{6RRCthm:n7} We have
	\begin{align}\label{rrc:eqn18}
		& q^{5/2}\(e^{2ix}\T_1^5(x+\pi\tau|5\tau)-e^{-2ix}\T_1^5(x-\pi\tau|5\tau)\)\\
		&\qquad \qquad  \times \(e^{4iy}\T_1^5(y+2\pi\tau|5\tau)-e^{-4iy}\T_1^5(y-2\pi\tau|5\tau)\)\nonumber\\
		&-q^{5/2}\(e^{2iy}\T_1^5(y+\pi\tau|5\tau)-e^{-2iy}\T_1^5(y-\pi\tau|5\tau)\)\nonumber\\
		&\qquad \qquad   \times \(e^{4ix}\T_1^5(x+2\pi\tau|5\tau)-e^{-4ix}\T_1^5(x-2\pi\tau|5\tau)\)\nonumber\\
		&=
		\(10\eta^4(\tau)\eta^4(5\tau)+\frac{\eta^{10}(\tau)}{\eta^2(5\tau)}\) \frac{\T_1(x+y|5\tau)\T_1(x-y|5\tau)\T_1(2x|5\tau)\T_1(2y|5\tau)}{\T_1(x|5\tau)\T_1(y|5\tau)}.\nonumber
	\end{align}	
	\end{thm}

Setting $y=0$ and $x=\pi\tau$ in Theorem~\ref{6RRCthm:n7} and simplifying we conclude that
\begin{align*}
	&q^{\frac{7}{2}}\T_1^{10} (2\pi\tau|5\tau)-q^{\frac{1}{2}}\T_1^{10}(\pi\tau|5\tau)-\T_1^5(\pi\tau|5\tau)\T_1^5(2\pi\tau|5\tau)	\\
	&=\(10\eta^4(\tau)\eta^4(5\tau)+\frac{\eta^{10}(\tau)}{\eta^2(5\tau)}\)
	\T_1(\pi\tau|5\tau)\T_1(2\pi\tau|5\tau).
\end{align*}
Dividing both sides of the above equation by $\T_1^5(\pi\tau|5\tau)\T_1^5(2\pi\tau|5\tau)$ we deduce that
\begin{align*}
	&\frac{q^{\frac{7}{2}}\T_1^{5} (2\pi\tau|5\tau)}{\T_1^{5} (\pi\tau|5\tau)}
	-\frac{q^{\frac{1}{2}}\T_1^{5} (\pi\tau|5\tau)}{\T_1^{5} (2\pi\tau|5\tau)}
	-1	\\
	&=\(10\eta^4(\tau)\eta^4(5\tau)+\frac{\eta^{10}(\tau)}{\eta^2(5\tau)}\)
	\frac{1}{\T_1^4(\pi\tau|5\tau)\T_1^4(2\pi\tau|5\tau)}.\nonumber
\end{align*}
Using the infinite product representation of $\T_1$  in the above equation we arrive at the identity due to Ramanujan:
\begin{equation}\label{rrc:eqn19}
	R^{-5}(\tau)-R^5(\tau)=11+\frac{\eta^{6}(\tau)}{\eta^6(5\tau)}.	
\end{equation}
Both of the identities  in  \reff{rrc:eqn15} and \reff{rrc:eqn19} were found by Watson \cite{Watson1929:a} in Ramanujan's second  notebook \cite[pp.~ 265-267]{Ramanujan1957} and proved by him for the purpose of establishing some of Ramanujan's claims about the Rogers--Ramanujan continued fraction in his first two letters to Hardy \cite[pp.~xxvii, xxviii]{Raman1927}.  Our proofs of  \reff{rrc:eqn15} and \reff{rrc:eqn19} are different from that of Watson in \cite{Watson1929:a} and Berndt \cite[pp.~265--267]{Berndt1991}.  These two identities were used by Lewis and Liu \cite{LewisLiu1999} to give simple proofs of Eisenstein series identities due to Ramanujan.  These two identities were also used by \cite{Ramanathan} and \cite{BCZ1996} to give some special values of the Rogers--Ramanujan continued fraction.

Using the binomial theorem  we can easily find that for any complex number numbers $a$ and $b$, 
\begin{equation}\label{rrc:eqn20}
	a^5-b^5=(a-b)^5+5ab (a-b)^3+5(ab)^2 (a-b).
\end{equation}
 Setting $a=\T_1(\frac{2\pi}{5}|\tau)/\T_1(\frac{\pi}{5}|\tau)$ and 
$b=\T_1(\frac{\pi}{5}|\tau)/\T_1(\frac{2\pi}{5}|\tau)$ in the above equation and then substituting \reff{rrc:eqn13} and \reff{rrc:eqn17} in the resulting equation we immediately find that
\begin{equation}\label{rrc:eqn21}
	11+125\frac{\eta^6(5\tau)}{\eta^6(\tau)}
	=\(1+5\frac{\eta(25\tau)}{\eta(\tau)}\)^5+5\(1+5\frac{\eta(25\tau)}{\eta(\tau)}\)^3+5\(1+5\frac{\eta(25\tau)}{\eta(\tau)}\),
\end{equation}
which is equivalent to the following modular equation of degree five:
\begin{align}\label{rrc:eqn22}
	&\frac{\eta(25\tau)}{\eta(\tau)}+5\(\frac{\eta(25\tau)}{\eta(\tau)}\)^2
	+15\(\frac{\eta(25\tau)}{\eta(\tau)}\)^3+25\(\frac{\eta(25\tau)}{\eta(\tau)}\)^4+25\(\frac{\eta(25\tau)}{\eta(\tau)}\)^5\\
	&=\frac{\eta^6(5\tau)}{\eta^6(\tau)}.\nonumber
\end{align}
It seems that this modular equation  was first discovered by Kiepert  \cite[p.~277, Eq.(62)]{Kiepert1885} in 1885.

Using \reff{rrc:eqn15} and the modular transformation formula for the Dedekind-eta function in \reff{jabel:eqn21} one can easily prove the following amazing theorem due to Ramanujan (see \cite[pp.83--85]{Berndt1991} and \cite[pp.211-212]{Ramanathan} for details).
\begin{thm} \label{6RRCthm:n8}If $\alpha$ and $\beta$ are positive such that $\alpha \beta =1$, then we have
	\begin{equation}\label{rrc:eqn23}
		\left\{\frac{\sqrt{5}+1}{2}+R(i\alpha)\right\}\left\{\frac{\sqrt{5}+1}{2}+R(i\beta)\right\}=\frac{5+\sqrt{5}}{2}.
	\end{equation}		
\end{thm} 	 
Using the modular transformation formula for the Dedekind-eta function in \reff{jabel:eqn21} we find that
\[
\eta(-1/{5i})=\sqrt{5}\eta(5i).
\]
Setting $\tau=i$ in \reff{rrc:eqn15} and using the above equation in the resulting  equation, we deduce that
\begin{equation*}
	1/R(i)-R(i)=1+\sqrt{5}.
\end{equation*}
So that we have the Ramanujan formula
\begin{equation}\label{rrc:eqn24}
	R(i)=\sqrt{\frac{5+\sqrt{5}}{2}}-\frac{1+\sqrt{5}}{2}.
\end{equation}
Setting $\tau=-1/{\sqrt{5}i}$ in \reff{rrc:eqn19} and then using the modular transformation formula for the Dedekind-eta function, one can easily find that
\[
R^{-5}\(-\frac{1}{\sqrt{5}i}\)-R^5\(-\frac{1}{\sqrt{5}i}\)=11+5\sqrt{5}
=2\(\frac{1+\sqrt{5}}{2}\)^5.
\]
From this equation one can find that (see \cite[p.233]{Watson1929:b} for details)
\begin{equation}\label{rrc:eqn25}
	R\(-\frac{1}{\sqrt{5}i}\)=\(\frac{1+\sqrt{5}}{2}\)\sqrt[5]{5^{3/4}\(\frac{\sqrt{5}-1}{2}\)^{5/2}-1}.
\end{equation}	
If we choose $\alpha=\sqrt{5}$ and $\beta=1/{\sqrt{5}}$ in  Theorem~\ref{6RRCthm:n8}, we can immediately get Ramanujan's formula
\begin{equation}\label{rrc:eqn26}
	R(\sqrt{5}i)=\frac{\sqrt{5}}{1+\sqrt[5]{5^{3/4}\(\frac{\sqrt{5}-1}{2}\)^{5/2}-1}}-\frac{\sqrt{5}+1}{2}.
\end{equation}
Dividing both sides of \reff{rrc:eqn16} by $q^{5/4}$ and then letting $q\to 0$, we arrive at the following trigonometric identity.
\begin{prop} We have
	\begin{align}	\label{rrc:eqn27}
		&\(\sin^5(x+\frac{\pi}{5})-\sin^5 (x-\frac{\pi}{5})\)\(\sin^5(y+\frac{2\pi}{5})-\sin^5 (y-\frac{2\pi}{5})\)\\
		&-\(\sin^5(x+\frac{2\pi}{5})
		-\sin^5 (x-\frac{2\pi}{5})\)\(\sin^5(y+\frac{\pi}{5})-\sin^5 (y-\frac{\pi}{5})\)\nonumber\\
		&=\frac{125}{32} \cos x \cos y \sin (x+y) \sin (x-y).\nonumber
	\end{align}	
\end{prop}

We will end this section by proving the following proposition.
\begin{prop}\label{6RRCthm:n9} We have
	\begin{equation}\label{rrc:eqn28}
		\sum_{n=-\infty}^\infty (-1)^n
		\(\frac{n+1}{5}\) q^{\frac{(6n+1)^2}{24}}
		=\sqrt{\eta^2(\tau)+2{\eta(\tau)} {\eta(25\tau)}+5{\eta^2(25\tau)}}.
	\end{equation}		
\end{prop}

\begin{proof} If we specialize Proposition~\ref{dirichletgaussKiepert} to the case when $m=5$, we easily find that
	\begin{equation}\label{rrc:eqn29}
		\frac{\T_1(\frac{2\pi}{5}|\tau)}{\T_1(\frac{\pi}{5}|\tau)}+\frac{\T_1(\frac{\pi}{5}|\tau)}{\T_1(\frac{2\pi}{5}|\tau)}=
		\frac{\sqrt{5}}{\eta(\tau)} \sum_{n=-\infty}^\infty (-1)^n
		\(\frac{n+1}{5}\) q^{\frac{(6n+1)^2}{24}}.
	\end{equation}		
   Squaring  both sides of the above equation,  we arrive at
	\[
	\(\frac{\T_1(\frac{2\pi}{5}|\tau)}{\T_1(\frac{\pi}{5}|\tau)}+\frac{\T_1(\frac{\pi}{5}|\tau)}{\T_1(\frac{2\pi}{5}|\tau)}\)^2=
	\(\frac{\sqrt{5}}{\eta(\tau)} \sum_{n=-\infty}^\infty (-1)^n
	\(\frac{n+1}{5}\) q^{\frac{(6n+1)^2}{24}}\)^2.
	\]
	 Squaring  both sides of \reff{rrc:eqn13},  we deduce that
	\begin{equation*}
		\(\frac{\T_1(\frac{2\pi}{5}|\tau)}{\T_1(\frac{\pi}{5}|\tau)}-\frac{\T_1(\frac{\pi}{5}|\tau)}{\T_1(\frac{2\pi}{5}|\tau)}\)^2=\(1+5\frac{\eta(25\tau)}{\eta(\tau)}\)^2.
	\end{equation*}
	Taking the difference of the above two equations and simplifying  we conclude that
	\[
	\(\sum_{n=-\infty}^\infty (-1)^n
	\(\frac{n+1}{5}\) q^{\frac{(6n+1)^2}{24}}\)^2
	=\eta^2(\tau)\(1+2\frac{\eta(25\tau)}{\eta(\tau)}+5\frac{\eta^2(25\tau)}{\eta^2(\tau)}\).
	\]
	 Taking the square root on both sides of the above equation  we complete the proof of the proposition.
\end{proof}
%%%%%%%%%%%%%%%%%%%%%%%%%%%%%%%%%%%%%%%%%%%%%%%%%%%%%%%%%%%%%%%%%%%%
\section{A general theta function identity of degree $6$}
%%%%%%%%%%%%%%%%%%%%%%%%%%%%%%%%%%%%%%%%%%%%%%%%%%%%%%%%%%%%%%%%%%%%
By taking $f(z|\tau)=\T_1(z+y|\tau)\T_1(z-y|\tau)F(z|\tau)$ in  Theorem~\ref{liuaddthm} we can easily derive the following beautiful  theta function identity.
\begin{thm}\label{legendrethm} Suppose that $F(z|\tau)$ is an even entire function of $z$ which satisfies the functional equations $F(z)=F(z+\pi|\tau)=q^{3} e^{12iz}F(z+\pi\tau|\tau)$.  Then we have
	\begin{align}\label{addliu:eqn1}
		\frac{4F(x|\tau)}{\T_1^2(2x|\tau)}
		=\frac{F(0|\tau)}{\T_1^2(x|\tau)}+\frac{F(\frac{\pi}{2}|\tau)}{\T_2^2(x|\tau)}
		-\frac{q^{\frac{3}{4}}F(\frac{\pi+\pi\tau}{2}|\tau)}{\T_3^2(x|\tau)}
		-\frac{q^{\frac{3}{4}} F(\frac{\pi\tau}{2}|\tau)}{\T_4^2(x|\tau)}.
	\end{align}		
\end{thm}	
This theorem is equivalent to \cite[Theorem~1.1]{Liu2010Mysore},  in which we gave several applications of it.  Here we will give more applications of this theorem. 

By taking $F(z|\tau)
=\T_1^3(z-x|\tau)\T_1^3(z+x|\tau)$ in Theorem~\ref{legendrethm} and simplifying we easily find that
\begin{equation}\label{addliu:eqn2}
	\T_2^4(x|\tau)+\T_4^4(x|\tau)=\T_1^4(x|\tau)+\T_3^4(x|\tau).
\end{equation}

If we take $F(z|\tau)=\T_1(z+y|\frac{\tau}{3})\T_1(z-y|\frac{\tau}{3})$ in Theorem~\ref{legendrethm} and then replacing $\tau$ by $3\tau$ in the resulting equation we find that
\begin{equation}\label{addliu:eqn3}
	\frac{\T_2^2(y|\tau)}{\T_2^2(x|3\tau)}-\frac{\T_3^2(y|\tau)}{\T_3^2(x|3\tau)}+\frac{\T_4^2(y|\tau)}{\T_4^2(x|3\tau)}	
	=\frac{\T_1^2(y|\tau)}{\T_1^2(x|3\tau)}	+\frac{4\T_1(x-y|\tau)\T_1(x+y|\tau)}{\T_1^2(2x|3\tau)}.	
\end{equation}

By setting $y=0$ in the above equation and then letting $x\to 0,$ we find that
\begin{equation}\label{addliu:eqn4}
	\(\frac{\VT_2(\tau)} {\VT_2(3\tau)}\)^2-\(\frac{\VT_3(\tau)} {\VT_3(3\tau)}\)^2
	+\(\frac{\VT_4(\tau)} {\VT_4(3\tau)}\)^2=\frac{\eta^6(\tau)}{\eta^6(3\tau)}.
\end{equation}

Applying the imaginary transformations to the above equation we conclude that
\begin{equation}\label{addliu:eqn5}
	\(\frac{\VT_2(3\tau)}{\VT_2(\tau)}\)^2-\(\frac{\VT_3(3\tau)}{\VT_3(\tau)}\)^2
	+\(\frac{\VT_4(3\tau)}{\VT_4(\tau)}\)^2=9\frac{\eta^6(3\tau)}{\eta^6(\tau)}.
\end{equation}
These two theta identities are equivalent to the modular equations of Ramanujan in \cite[p.~230, Entry 5 (vii)]{Berndt1991}.

If we specialize Theorem~\ref{legendrethm} to the case when $F(z|\tau)$ is given by
\begin{align*}
F(z|\tau)&=\frac{\T_1^2(z-x|\tau)\T_1(z+y|\tau)\T_1(z-y|\tau)\T_1(3z+3x|3\tau)}{\T_1(z+x|\tau)}\\
\quad &+\frac{\T_1^2(z+x|\tau)\T_1(z+y|\tau)\T_1(z-y|\tau)\T_1(3z-3x|3\tau)}{\T_1(z-x|\tau)},
\end{align*}
we conclude that
\begin{align}\label{addliu:eqn6}
	\T_1^2(y|\tau) \frac{\T_1(3x|3\tau)}{\T_1(x|\tau)}+\T_2^2(y|\tau) \frac{\T_2(3x|3\tau)}{\T_2(x|\tau)}-\T_3^2(y|\tau) \frac{\T_3(3x|3\tau)}{\T_3(x|\tau)}	+\T_4^2(y|\tau) \frac{\T_4(3x|3\tau)}{\T_4(x|\tau)}\\
	=-6\T_1(x-y|\tau)\T_1(x+y|\tau)\frac{\eta^3(3\tau)}{\eta^3(\tau)}.\nonumber
\end{align}
Setting $y=x$ in the above equation and noting that $\T_1(0|\tau)=0$ we immediately deduce that
	\begin{equation}\label{addliu:eqn7}
	\T_1(x|\tau)\T_1(3x|3\tau)+\T_2(x|\tau)\T_2(3x|3\tau)+\T_4(x|\tau)\T_4(3x|3\tau)
	=\T_3(x|\tau)\T_3(3x|3\tau).
\end{equation}
If we specialize the above equation to the case when $x=0$, we arrive at the Legendre identity \cite[p.~230, Entry 5(ii)]{Berndt1991}
\begin{equation}\label{addliu:eqn8}
	\VT_2(\tau)\VT_2(3\tau)+\VT_4(\tau)\VT_4(3\tau)=\VT_3(\tau)\VT_3(3\tau).
\end{equation}

Appealing to the case $n=3$ of   the multiplication formula in \reff{jabel:eqn16} one can easily find that
\begin{equation}\label{addliu:eqn9}
	 \T_1\(\frac{\pi}{3}\Big|\tau\)=\sqrt{3}\eta(3\tau)\quad \text{and}\quad
	\T_j^2 \(\frac{\pi}{3}\Big|\tau\)=\frac{\eta^3(\tau)\VT_j(3\tau)}{\eta(3\tau)\VT_j(\tau)} \quad \text{for}\quad j=2, 3, 4.
\end{equation}

Setting $x=\pi/3$ in \reff{addliu:eqn7} and appealing to the above equation we deduce that
\begin{equation}\label{addliu:eqn10}
	\sqrt{\frac{\VT_4^3(3\tau)}{\VT_4(\tau)}}-\sqrt{\frac{\VT_2^3(3\tau)}{\VT_2(\tau)}}=\sqrt{\frac{\VT_3^3(3\tau)}{\VT_3(\tau)}}.	
\end{equation}

Applying the imaginary transformations to the above formula we can deduce that
\begin{equation}\label{addliu:eqn11}
	\sqrt{\frac{\VT_4^3(\tau)}{\VT_4(3\tau)}}-\sqrt{\frac{\VT_2^3(\tau)}{\VT_2(3\tau)}}=\sqrt{\frac{\VT_3^3(\tau)}{\VT_3(3\tau)}}.	
\end{equation}
The above two theta function identities are equivalent to the modular equations \cite[p.~230, Entry 5(i)]{Berndt1991}.

Appealing to the case $n=3$ of   the multiplication formula in \reff{jabel:eqn17} one can easily find that
\begin{equation}\label{addliu:eqn12}
	\T_1(\pi\tau|3\tau)=iq^{-1/6}\eta(\tau)~ \text{and}~
	\T_j^2 \(\pi\tau|3\tau\)=q^{-1/3}\frac{\eta^3(3\tau)\VT_j(\tau)}{\eta(\tau)\VT_j(3\tau)} \quad \text{for}\quad j=2, 3, 4.
\end{equation}

Replacing $\tau$ by $3\tau$ in \reff{addliu:eqn7} and then taking $x=\pi\tau$ and finally using \reff{addliu:eqn12} in the resulting equation, we deduce that
\begin{equation}\label{addliu:eqn13}
	\sqrt{\frac{\VT_2(\tau)} {\VT_2(9\tau)}}-\sqrt{\frac{\VT_3(\tau)} {\VT_3(9\tau)}}+\sqrt{\frac{\VT_4(\tau)} {\VT_4(9\tau)}}=\sqrt{\frac{\eta^3(\tau)} {\eta^3(9\tau)}},
\end{equation}
which is equivalent to \cite[p,~ 352, Entry 3 (x)]{Berndt1991}.

Applying the imaginary transformations to the above formula we find that
\begin{equation}\label{addliu:eqn14}
	\sqrt{\frac{\VT_2(9\tau)}{\VT_2(\tau)}}-\sqrt{\frac{\VT_3(9\tau)}{\VT_3(\tau)}}+\sqrt{\frac{\VT_4(9\tau)}{\VT_4(\tau)}}=3\sqrt{\frac{\eta^3(9\tau)}{\eta^3(\tau)}},
\end{equation}
which is equivalent to \cite[p, 352, Entry 3 (xi)]{Berndt1991}. 

By taking $x=\pi/3$ and $y=0$ in \reff{addliu:eqn6} and then using \reff{addliu:eqn9} to simplify the resulting equation we obtain that
\begin{equation}\label{addliu:eqn15}
	\sqrt{\VT_2^5(\tau)\VT_2(3\tau)}+\sqrt{\VT_3^5(\tau)\VT_3(3\tau)}
	-\sqrt{\VT_4^5(\tau)\VT_4(3\tau)}=18\sqrt{\frac{\eta^9(3\tau)}{\eta^3(\tau)}}.
\end{equation}

If we let $x\to 0$ in \reff{addliu:eqn6}, then we immediately conclude that
	\begin{align}\label{addliu:eqn16}
		\T_2^2(y|\tau) \frac{\VT_2(3\tau)}{\VT_2(\tau)}-\T_3^2(y|\tau) \frac{\VT_3(3\tau)}{\VT_3(\tau)}+\T_4^2(y|\tau) \frac{\VT_4(3\tau)}{\VT_4(\tau)}
		=3\T_1^2(y|\tau)\frac{\eta^3(3\tau)}{\eta^3(\tau)}.
	\end{align}

If we take $y=\pi/2,~(\pi+\pi\tau)/2$ and $(\pi\tau)/2$ respectively in \reff{addliu:eqn16} and  use the fact that
\[
\T_2(\pi/2|\tau)=\T_3((\pi+\pi\tau)/2|\tau)=\T_4((\pi\tau)/2|\tau)=0
\]
in the resulting equation we can obtain the following theta function identities:
\begin{equation}\label{addliu:eqn17}
	\begin{split}
		3\VT_2(\tau)\VT_2(3\tau)=\frac{\VT_3^3(\tau)}{\VT_3(3\tau)}-\frac{\VT_4^3(\tau)}{\VT_4(3\tau)},\\
		3\VT_3(\tau)\VT_3(3\tau)=\frac{\VT_2^3(\tau)}{\VT_2(3\tau)}-\frac{\VT_4^3(\tau)}{\VT_4(3\tau)},\\
		3\VT_4(\tau)\VT_4(3\tau)=\frac{\VT_2^3(\tau)}{\VT_2(3\tau)}-\frac{\VT_3^3(\tau)}{\VT_3(3\tau)}.	
	\end{split}	
\end{equation} 
Replacing $\tau$ by $-1/{3\tau}$ in the above equations and then using the imaginary  transformations formulas in \reff{jabel:eqn20} we have 
\begin{equation}\label{addliu:eqn18}
	\begin{split}
		\VT_2(\tau)\VT_2(3\tau)=\frac{\VT_4^3(3\tau)}{\VT_4(\tau)}-\frac{\VT_3^3(3\tau)}{\VT_3(\tau)},\\
		\VT_3(\tau)\VT_3(3\tau)=\frac{\VT_4^3(3\tau)}{\VT_4(\tau)}-\frac{\VT_2^3(3\tau)}{\VT_2(\tau)},\\
		\VT_4(\tau)\VT_4(3\tau)=\frac{\VT_3^3(3\tau)}{\VT_3(\tau)}-\frac{\VT_2^3(3\tau)}{\VT_2(\tau)}.	
	\end{split}	
\end{equation} 
The above six theta identities can be found in \cite[pp.~ 1105--1106 ]{Shen1994PAMS}.

Let us take $F(z|\tau)=\T_1(7z|7\tau)/{\T_1(z|\tau)}$ in Theorem~\ref{legendrethm}. Then  we conclude that

\begin{align}\label{addliu:eqn19}
	&\frac{7\VT_1'(7\tau)}{\VT_1'(\tau)\T_1^2(x|\tau)}-	\frac{\VT_2(7\tau)}{\VT_2(\tau)\T_2^2(x|\tau)}+\frac{\VT_3(7\tau)}{\VT_3(\tau)\T_3^2(x|\tau)}-\frac{\VT_4(7\tau)}{\VT_4(\tau)\T_4^2(x|\tau)}\\
	&=\frac{4\T_1(7x|7\tau)}{\T_1(x|\tau)\T_1^2(2x|\tau)}.\nonumber
\end{align}	
Putting $x=\pi/3$ in \reff{addliu:eqn19} and using \reff{addliu:eqn9}  and simplifying we conclude that
\begin{equation}\label{addliu:eqn20}
	\frac{\VT_2(7\tau)}{\VT_2(3\tau)}-\frac{\VT_3(7\tau)}{\VT_3(3\tau)}+
	\frac{\VT_4(7\tau)}{\VT_4(3\tau)}
	=\frac{7\eta^3(7\tau)}{3\eta^3(3\tau)}-\frac{4\eta^3(\tau) \eta(21\tau)}{3\eta^4(3\tau)}.
\end{equation}

Applying the imaginary transformations to  both sides of the above equation yields
\begin{equation}\label{addliu:eqn21}
	\frac{\VT_2(3\tau)} {\VT_2(7\tau)}-\frac{\VT_3(3\tau)}{\VT_3(7\tau)}+
	\frac{\VT_4(3\tau)} {\VT_4(7\tau)}
	=\frac{\eta^3(3\tau)}{\eta^3(7\tau)}-4\frac{\eta(\tau) \eta^3(21\tau) }{\eta^4(7\tau)}.
\end{equation}
 
If we choose $F(z|\tau)=\T_1(z|\frac{\tau}{7})/{\T_1(z|\tau)}$ in Theorem~\ref{legendrethm}, then by some simple calculations we find that
\begin{align}\label{addliu:eqn22}
	&\frac{\VT_1'(\tau)}{\VT_1'(7\tau)\T_1^2(x|7\tau)}+	\frac{\VT_2(\tau)}{\VT_2(7\tau)\VT_2^2(x|7\tau)}-\frac{\VT_3(\tau)}{\VT_3(7\tau)\VT_3^2(x|7\tau)}+\frac{\VT_4(\tau)}{\VT_4(7\tau)\VT_4^2(x|7\tau)}\\
	&=\frac{4\T_1(x|\tau)}{\T_1(x|7\tau)\T_1^2(2x|7\tau)}.\nonumber
\end{align}

Letting  $x=\pi/3$ in \reff{addliu:eqn22},   making use of  \reff{addliu:eqn9},  and simplifying  we arrive at
\begin{equation}\label{addliu:eqn23}
	\frac{\VT_2(\tau)}{\VT_2(21\tau)}-\frac{\VT_3(\tau)}{\VT_3(21\tau)}	
	+\frac{\VT_4(\tau)}{\VT_4(21\tau)}
	=\frac{4\eta(3\tau)\eta^3(7\tau)}{3\eta^4(21\tau)}-\frac{\eta^3(\tau)}{3\eta^3(21\tau)}.
\end{equation}

Applying the imaginary transformations to  both sides of the above equation gives
\begin{equation}\label{addliu:eqn21}
	\frac{\VT_2(21\tau)}{\VT_2(\tau)}-\frac{\VT_3(21\tau)}{\VT_3(\tau)}	
	+\frac{\VT_4(21\tau)}{\VT_4(\tau)}
	=\frac{4\eta^3(3\tau)\eta(7\tau)}{\eta^4(\tau)}-7\frac{\eta^3(21\tau)}{\eta^3(\tau)}.
\end{equation}
%%%%%%%%%%%%%%%%%%%%%%%%%%%%%%%%%%%%%%%%%%%%%%%%%%%%%%%%%%%%%%%%%%%%%%%%%%%%%
\section{ Some applications of Theorem~\ref{liuaddthm:lim}}
%%%%%%%%%%%%%%%%%%%%%%%%%%%%%%%%%%%%%%%%%%%%%%%%%%%%%%%%%%%%%%%%%%%%%%%%%%%%%
The following trigonometric identity first appeared in Ramanujan's paper in 1916 \cite[Eq.(18)]{Ramanujan1916} without proof. He used this identity to get some recurrence relations for Eisenstein series.  For the application of this formula to the representations of integers as sums of squares and as sums of triangular numbers, please refer to \cite{Liu2001Squares} and \cite{Liu2003RamJ}.
\begin{prop}[Ramanujan]\label{ramtrithm} We have 
	\begin{align}\label{Ei:eqn1}
		&\left\{\frac{1}{8} \cot^2 x+\frac{1}{12}+\sum_{n=1}^\infty \frac{nq^n}{1-q^n} \(1-\cos 2nx\)\right\}^2\\
		&=\left\{\frac{1}{8}\cot^2 x+\frac{1}{12}\right\}^2
		+\frac{1}{12}\sum_{n=1}^\infty \frac{n^3q^n}{1-q^n} (5+\cos 2nx).\nonumber
	\end{align}
\end{prop}
Now we will use  Theorem~\ref{liuaddthm:lim} to prove the above proposition.
\begin{proof}
	Let us take $f(z|\tau)\T_1^2(z|\tau)=\T_1^2(2z|\tau)\T_1(z+x|\tau)\T_1(z-x|\tau)$ in Theorem~\ref{liuaddthm:lim}.  
	Since $0,~ {\pi}/{2},~ (\pi+\pi\tau)/{2}$ and $ (\pi\tau)/2 $ 
	are zeros of $\T_1(2z|\tau)$,  we easily find that $f(\pi/2|\tau)=f((\pi+\pi\tau)/2|\tau)=f((\pi\tau)/2|\tau)=0, ~ f(0|\tau)=-4\T_1^2(u|\tau)$. Substituting these values of $f$ into \reff{jabel:eqn24} we find that
	\begin{equation}\label{Ei:eqn2}
		\(8L(\tau)+3(\log f)''(0|\tau)\)^2+8M(\tau) +3 (\log f)^{(4)}(0|\tau)=0.
	\end{equation}

	Now we begin to compute $(\log f)''(0|\tau)$ and $(\log f)^{(4)}(0|\tau)$. Using the asymptotic expansion in \reff{jabel:eqn14} we find that near $z=0$, 
	\begin{align}\label{Ei:eqn3}
		&(\log f)'(z|\tau)\\
		&=4\(\log \T_1\)'(2z|\tau)-2\(\log \T_1\)'(z|\tau)
		+\(\log \T_1\)'(z+x|\tau)+\(\log \T_1\)'(z-x|\tau) \nonumber\\
		&=-2L(\tau)z-\frac{2}{3}M(\tau)z^3+\(\log \T_1\)'(z+x|\tau)+\(\log \T_1\)'(z-x|\tau).\nonumber
	\end{align}
	It follows that
	\begin{align*}
		(\log f)''(0|\tau)&=-2L(\tau)+2 \(\log \T_1\)''(x|\tau), \\
		(\log f)^{(4)}(0|\tau)&=-4M(\tau)+2 \(\log \T_1\)^{(4)}(x|\tau).
	\end{align*}

	Substituting the above two equations into \reff{Ei:eqn3} and then dividing both sides of the equation by $4$ we conclude that
	\begin{equation}\label{Ei:eqn4}
		\(L(\tau)+3 \(\log \T_1\)''(x|\tau)\)^2=M(\tau)-\frac{3}{2}\(\log \T_1\)^{(4)}(x|\tau).
	\end{equation}

	By substituting the trigonometric series expansion  for the partial logarithmic derivative of $\T_1(z|\tau)$ with respect to $z$ in \reff{jabel:eqn13} into the above equation and simplifying we complete the proof of Proposition~\ref{ramtrithm}.
\end{proof}
With the help of \reff{jabel:eqn15} we know that \reff{Ei:eqn4} is equivalent to the differential equation
\begin{equation}\label{Ei:eqn5}
	\wp''(z|\tau)=6\wp^2(z|\tau)-\frac{2}{3}M(\tau).
\end{equation}
Differentiation of the differential equation for $\wp(z|\tau)$ in \reff{jabel:eqn30} also yields  the above differential equation. Conversely, integration of the above equation we can also obtain the differential equation in \reff{jabel:eqn30}.

If we replace $x$ by $x+\pi/2$ in \reff{Ei:eqn4} and appeal to Proposition~\ref{halfperiods}, we deduce that
\begin{equation}\label{Ei:eqn4a}
	\(L(\tau)+3 \(\log \T_4\)''(x|\tau)\)^2=M(\tau)-\frac{3}{2}\(\log \T_4\)^{(4)}(x|\tau).
\end{equation}
Substituting the Fourier series for $\(\log \T_4\)''(x|\tau)$ and $\(\log \T_4\)^{(4)}(x|\tau)$ into the above equation and then replacing $\tau$ by 
$2\tau$ we conclude that \cite[Theorem~10]{Liu2003RamJ}
\begin{align}\label{Ei:eqn4b}
	&\(1+24\sum_{n=1}^\infty \frac{q^{2n}}{1-q^{2n}}+24\sum_{n=1}^\infty \frac{q^n}{1-q^{2n}} \cos 2nu\)^2\\
	&=1+240\sum_{n=1}^\infty \frac{n^3 q^{2n}}{1-q^{2n}} +48
	\sum_{n=1}^\infty \frac{n^3q^n}{1-q^{2n}} \cos 2nu. \nonumber
\end{align}

\begin{prop}\label{Eisen:n1} Let $L(\tau)$ and $M(\tau)$ be  the first two  Eisenstein series defined by \reff{jabel:eqn9}. Then we have
	\begin{align}\label{Ei:eqn6}
		&(9L(9\tau)-L(\tau))^2 +\frac{1}{5}\(42M(9\tau)-2M(\tau)\)\\
		&=\frac{72\VT_1'(9\tau)^5}{\VT_1'(\tau)}
		\(\frac{\VT_2(\tau)}{\VT_2^5(9\tau)}-\frac{\VT_3(\tau)}{\VT_3^5(9\tau)}+\frac{\VT_4(\tau)}{\VT_4^5(9\tau)}\).\nonumber
	\end{align}	
\end{prop}
\begin{proof} Let us take $f(z|\tau)=\T_1(z|\frac{\tau}{9})/{\T_1(z|\tau)}$ in Theorem~\ref{liuaddthm:lim}.  It is easily seen that $f(0|\tau)=\VT_1'(\frac{\tau}{9})/{\VT'(\tau)}$,  and using Proposition~\ref{halfperiods} we find that
	\begin{equation}\label{Ei:eqn7}
		f\(\frac{\pi}{2}\Big|\tau\)=\frac{\VT_2(\frac{\tau}{9})}{\VT_2(\tau)}, ~
		f\(\frac{\pi+\pi\tau}{2}\Big|\tau\)=\frac{\VT_3(\frac{\tau}{9})}{q\VT_3(\tau)},~f\(\frac{\pi \tau}{2}\Big|\tau\)=\frac{\VT_4(\frac{\tau}{9})}{q\VT_4(\tau)}.
	\end{equation}

	Using the asymptotic expansion for the partial logarithmic derivative of 
	$\T_1(z|\tau)$ with respect to $z$ near $z=0$ in \reff{jabel:eqn14}, we find that near $z=0$,
	\begin{align*}
		\(\log f\)' (z|\tau)&=(\log \T_1)'(z|{\tau}/{9})-(\log \T_1)'(z|\tau)\\
		&=\frac{1}{3}\(L(\tau)-L({\tau}/{9})\)z+\frac{1}{45}\(M(\tau)-M({\tau}/{9})\)z^3+O(z^5).
	\end{align*}
	It follows that
	\begin{align*}
		\(\log f\)'' (0|\tau)&=\frac{1}{3}\(L(\tau)-L({\tau}/{9})\),\\
		\(\log f\)^{(4)}(0|\tau)&=\frac{2}{15}\(M(\tau)-M({\tau}/{9})\).
	\end{align*}

	Substituting the above equations and \reff{Ei:eqn7} into \reff{jabel:eqn24} and then replacing $\tau$ by $9\tau$ in the resulting equation we complete the proof of Proposition~\ref{Eisen:n1}.	
\end{proof}
\begin{prop}\label{Eisen:n2} We have
	\begin{align}\label{Ei:eqn8}
		&405(9L(9\tau)-L(\tau))^2 +42M(\tau)-13122M(9\tau)\\
		&=\frac{40\VT_1'(\tau)^5}{\VT_1'(9\tau)}
		\(\frac{\VT_2(9\tau)}{\VT_2^5(\tau)}-\frac{\VT_3(9\tau)}{\VT_3^5(\tau)}+\frac{\VT_4(9\tau)}{\VT_4^5(\tau)}\).\nonumber
	\end{align}	
\end{prop}
\begin{proof} Using Proposition~\ref{doubleperiods} we can verify that $f(z|\tau)=\T_1(9z|9\tau)/{\T_1(z|\tau)}$ satisfies the conditions of Theorem~\ref{liuaddthm:lim}. By a direct computation we deduce that
	$f(0|\tau)=9\VT_1'(9\tau)/{\VT_1'(\tau)}$ and	
	\begin{equation}\label{Ei:eqn9}
		f\(\frac{\pi}{2}\Big|\tau\)=\frac{\VT_2(9\tau)}{\VT_2(\tau)}, ~
		f\(\frac{\pi+\pi\tau}{2}\Big|\tau\)=\frac{\VT_3(9\tau)}{q\VT_3(\tau)},~f\(\frac{\pi \tau}{2}\Big|\tau\)=\frac{\VT_4(9\tau)}{q\VT_4(\tau)}.
	\end{equation}	

	With the help of  the asymptotic expansion for the partial logarithmic derivative of 
	$\T_1(z|\tau)$ with respect to $z$ near $z=0$ in \reff{jabel:eqn14}, we find that near $z=0$,
	\begin{align*}
		(\log f)'(z|\tau)&= 9(\log \T_1)'(9z|9\tau)-(\log \T_1)'(z|\tau)\\
		&=\frac{1}{3}\(L(\tau)-81L(9\tau)\)z+\frac{1}{45}\(M(\tau)-6561M(9\tau)\)z^3+O(z^5).
	\end{align*}
	It follows that
	\begin{align*}
		(\log f)''(0|\tau)&=\frac{1}{3}\(L(\tau)-81L(9\tau)\),\\
		(\log f)^{(4)}(0|\tau)&=\frac{2}{15}\(M(\tau)-6561M(9\tau)\).
	\end{align*}

	Substituting the above equations and \reff{Ei:eqn9} into \reff{jabel:eqn24}  and simplifying we complete the proof of Proposition~\ref{Eisen:n2}.	
\end{proof}

\begin{prop}\label{Eisen:n3}  We have
	\begin{align}\label{Ei:eqn10}
		&125(L(\tau)-5L(5\tau))^2 +11M(\tau)-625M(5\tau)\\
		&=\frac{18\VT_1'(\tau)^6}{5\VT_1'(5\tau)^2}
		\(\frac{\VT_2^2(5\tau)}{\VT_2^6(\tau)}-\frac{\VT_3^2(5\tau)}{\VT_3^6(\tau)}+\frac{\VT_4^2(5\tau)}{\VT_4^6(\tau)}\).\nonumber
	\end{align}	
\end{prop}
\begin{proof} If we specialize $f(z|\tau)$ in Theorem~\ref{liuaddthm:lim} to the case when $f(z|\tau)=\T_1^2(5z|5\tau)/{\T_1^2(z|\tau)}$, then we have 
	$f(0|\tau)=25\VT_1'(5\tau)^2/{\VT_1'(\tau)^2}$ and 	 
	\begin{align}\label{Ei:eqn11}
		f\(\frac{\pi}{2}\Big|\tau\)=\frac{\VT_2^2(5\tau)}{\VT_2^2(\tau)}, ~
		f\(\frac{\pi+\pi\tau}{2}\Big|\tau\)=\frac{\VT_2^2(5\tau)}{q\VT_2^2(\tau)},~f\(\frac{\pi \tau}{2}\Big|\tau\)=\frac{\VT_4^2(5\tau)}{q\VT_4(\tau)}.
	\end{align}	

	Using the asymptotic expansion for the partial logarithmic derivative of 
	$\T_1(z|\tau)$ with respect to $z$ near $z=0$ in \reff{jabel:eqn14}, we find that near $z=0$,
	\begin{align*}
		\(\log f\)'(z|\tau)&=10(\log \T_1)'(5z|5\tau)-2(\log \T_1)'(z|\tau)\\
		&=\frac{2}{3}\(L(\tau)-25L(5\tau)\)z+\frac{2}{45}\(M(\tau)-625M(5\tau)\)z^3+O(z^5).
	\end{align*}
	It follows that
	\begin{align*}
		\(\log f\)''(0|\tau)&=\frac{2}{3}\(L(\tau)-25M(5\tau)\),\\
		\(\log f\)^{(4)}(0|\tau)&=\frac{4}{15}\(M(\tau)-625M(5\tau)\).
	\end{align*}

	Substituting the above equations and \reff{Ei:eqn11} into \reff{jabel:eqn24}  we complete the proof of Proposition~\ref{Eisen:n3}.
\end{proof}	
\begin{prop}\label{Eisen:n4} We have
	\begin{align}\label{Ei:eqn12}
		&5\(L(\tau)-6L(3\tau)+9L(9\tau)\)^2-M(\tau)+12M(3\tau)-81M(9\tau)\\
		&=\frac{10\VT_1'(3\tau)^8}{\VT_1'(\tau)^2 \VT_1'(9\tau)^2}
		\(\frac{\VT_2^2(\tau)\VT_2^2(9\tau)}{\VT_2^8(3\tau)}-\frac{\VT_3^2(\tau)\VT_3^2(9\tau)}{\VT_3^8(3\tau)}+\frac{\VT_4^2(\tau)\VT_4^2(9\tau)}{\VT_4^8(3\tau)}\).\nonumber
	\end{align}	
\end{prop}
\begin{proof} With the help of  Proposition~\ref{doubleperiods} it is easily verified  that the function 
	\[
	f(z|\tau)=\frac{\T_1^2(3z|3\tau)\T_1^2(z|\frac{\tau}{3})}{\T_1^4(z|\tau)}
	\]	
	satisfies the conditions of Theorem~\ref{liuaddthm:lim}.  A direct computation shows that
	\begin{align}\label{Ei:eqn13}
		&f(0|\tau)=\frac{9\VT_1'(3\tau)^2 \VT_1'(\frac{\tau}{3})^2}{\VT_1'(\tau)^4},~f\(\frac{\pi}{2}\Big|\tau\)=\frac{\VT_2^2(3\tau)\VT_2^2(\frac{\tau}{3})}{ \VT_2^4(\tau)},\\
		&f\(\frac{\pi+\pi\tau}{2}\Big|\tau\)=\frac{\VT_3^2(3\tau)\VT_3^2(\frac{\tau}{3})}{q \VT_3^4(\tau)},  ~f\(\frac{\pi\tau}{2}\Big|\tau\)=\frac{\VT_4^2(3\tau)\VT_4^2(\frac{\tau}{3})}{q \VT_4^4(\tau)}. \label{Ei:eqn14}
	\end{align}

	Appealing to the asymptotic expansion for the partial logarithmic derivative of $\T_1(z|\tau)$ with respect to $z$ near $z=0$ in \reff{jabel:eqn24}, we deduce  that near $z=0$,
	\begin{align*}
		(\log f)'(z|\tau)&=6(\log \T_1)'(3z|3\tau)
		+2(\log \T_1)'\(z \Big|\frac{\tau}{3}\)-4(\log \T_1)'(z|\tau)\\
		&=\frac{2}{3}\(2L(\tau)-L\(\frac{\tau}{3}\)-9L(3\tau)\)z\\
		&\quad +\frac{2}{45}\(2M(\tau)-M\(\frac{\tau}{3}\)-81M(3\tau)\)z^3+O(z^5).
	\end{align*}
	It follows that
	\begin{align*}
		(\log f)''(0|\tau)&=\frac{2}{3}\(2L(\tau)-L\(\frac{\tau}{3}\)-9L(3\tau)\),\\
		(\log f)^{(4)}(0|\tau)&=\frac{4}{15}\(2M(\tau)-M\(\frac{\tau}{3}\)-81M(3\tau)\).
	\end{align*}

	Substituting the above equations and \reff{Ei:eqn13} and \reff{Ei:eqn14} into \reff{jabel:eqn24} and then replacing $\tau$ by $3\tau$  we complete the proof of Proposition~\ref{Eisen:n4}.	
\end{proof}

\begin{prop}\label{Eisen:n5}  We have
	\begin{align}\label{Ei:eqn15}
		&5\(10L(5\tau)-L(\tau)-25L(25\tau)\)^2+44M(5\tau)-2M(\tau)-1250M(25\tau)\\
		=&\frac{72\VT_1'(5\tau)^6}{\VT_1'(\tau)\T_1'(25\tau)}
		\(\frac{\VT_2(\tau)\VT_2(25\tau)}{\VT_2^6(5\tau)}-\frac{\VT_3(\tau)\VT_3(25\tau)}{\VT_3^6(5\tau)}+\frac{\VT_4(\tau)\VT_4(25\tau)}{\VT_4^6(5\tau)}\).\nonumber
	\end{align}		
\end{prop}
We can derive this proposition by applying  Theorem~\ref{liuaddthm:lim} to the case when 
\[
f(z|\tau)=\frac{\T_1(5z|5\tau)\T_1(z|\frac{\tau}{5})}{\T_1^2(z|\tau)}.
\]
\begin{prop}\label{Eisen:n6}  We have
	\begin{align}\label{Ei:eqn16}
		&5\(10L(21\tau)-L(3\tau)-L(7\tau)\)^2+44M(5\tau)-2M(3\tau)-2M(7\tau)\\
		=&\frac{360\VT_1'(21\tau)^6}{\VT_1'(3\tau)\T_1'(7\tau)}
		\(\frac{\VT_2(3\tau)\VT_2(7\tau)}{\VT_2^6(21\tau)}-\frac{\VT_3(3\tau)\VT_3(7\tau)} {\VT_3^6(21\tau)}+\frac{\VT_4(3\tau)\VT_4(7\tau)}{\VT_4^6(21\tau)}\).\nonumber
	\end{align}		
\end{prop}
This proposition can be obtained by applying Theorem~\ref{liuaddthm:lim} to the function
\[
f(z|\tau)=\frac{\T_1(z|\frac{\tau}{7})\T_1(z|\frac{\tau}{3})}{\T_1^2(z|\tau)}.
\]
\begin{prop}\label{Eisen:n7}  We have
	\begin{align}\label{Ei:eqn17}
		&5 \(7L(15\tau)+L(3\tau)+L(5\tau)-L(\tau)\)^2\\	&+38M(15\tau)+2M(5\tau)+2M(3\tau)-2M(\tau)\nonumber\\
		&=\frac{360\VT_1'(15\tau)^3\VT_1'(3\tau)\VT_1'(5\tau)}{\VT_1'(\tau)}\nonumber\\	
		&\times \(\frac{\VT_2(\tau)}{\VT_2(3\tau)\VT_2(5\tau)\VT_2^{3}(15\tau)}
		-\frac{\VT_3(\tau)}{\VT_3(3\tau)\VT_3(5\tau)\VT_3^{3}(15\tau)}
		+\frac{\VT_4(\tau)}{\VT_4(3\tau)\VT_4(5\tau)\VT_4^{3}(15\tau)}\).\nonumber
	\end{align}	
\end{prop}
\begin{proof} Appealing to Proposition~\ref{doubleperiods}  we can  verify  that the entire function 
	\[
	f(z|\tau)=\frac{\T_1(z|\frac{\tau}{15})\T_1(z|\tau)}{\T_1(z|\frac{\tau}{3})\T_1(z|\frac{\tau}{5})}
	\]	
	satisfies the conditions of Theorem~\ref{liuaddthm:lim}.  A direct computation shows that
	\begin{align*}
		f(0|\tau)&=\frac{\VT_1'(\frac{\tau}{15})\VT_1'(\tau)}{\VT_1'(\frac{\tau}{3})\VT_1'(\frac{\tau}{5})},~~	f\(\frac{\pi\tau}{2}|\tau\)=\frac{\VT_4(\frac{\tau}{15})\VT_4(\tau)}{q\VT_4(\frac{\tau}{3})\VT_4(\frac{\tau}{5})},\\
		f\(\frac{\pi}{2}|\tau\)&=\frac{\VT_2(\frac{\tau}{15})\VT_2(\tau)}{\VT_2(\frac{\tau}{3})\VT_2(\frac{\tau}{5})},~~
		f\(\frac{\pi+\pi\tau}{2}|\tau\)=\frac{\VT_3(\frac{\tau}{15})\VT_3(\tau)}{q\VT_3(\frac{\tau}{3})\VT_3(\frac{\tau}{5})},
	\end{align*}
	and 
	\begin{align*}
		(\log f)''(0|\tau)&=\frac{1}{3}\(L\(\frac{\tau}{3}\)+L\(\frac{\tau}{5}\)-L(\tau)-L\(\frac{\tau}{15}\)\),\\	
		(\log f)^{(4)}(0|\tau)&=\frac{2}{15}\(M\(\frac{\tau}{3}\)+M\(\frac{\tau}{5}\)-M(\tau)-M\(\frac{\tau}{15}\)\).
	\end{align*}

	Substituting the above equations into Theorem~\ref{liuaddthm:lim} and then replacing $\tau$ by $15\tau$, we complete the proof of Proposition~\ref{Eisen:n7}.
\end{proof}
By choosing $f(z|\tau)=\T_1^4(3z|3\tau)/{\T_1^4(z|\tau)}$ in Theorem~\ref{liuaddthm:lim} and simplifying we can deduce that
\begin{align}\label{Eisenram:eqn1}
	&90\(L(\tau)-3L(3\tau)\)^2 +6M(\tau)-81M(3\tau)\\
	&=\frac{5\VT_1'(\tau)^8}{\VT_1'(3\tau)^4}
	\(\frac{\VT_2^4(3\tau)}{\VT_2^8(\tau)}-\frac{\VT_3^4(3\tau)}{\VT_3^8(\tau)}+\frac{\VT_4^4(3\tau)}{\VT_4^8(\tau)}\).\nonumber
\end{align}
%%%%%%%%%%%%%%%%%%%%%%%%%%%%%%%%%%%%%%%%%%%%%%%%%%%%%%%%%%%%%%%%%%
\section{More applications of Theorem~\ref{liuaddthm}}
%%%%%%%%%%%%%%%%%%%%%%%%%%%%%%%%%%%%%%%%%%%%%%%%%%%%%%%%%%%%%%%%
\begin{thm}\label{liuapp:eqn1} Suppose that $f(z|\tau)$ is an even entire function of $z$ which satisfies the  functional equations $f(z|\tau)=f(z+\pi|\tau)=q^4 e^{16iz}f(z+\pi\tau|\tau)$. Then we have
\begin{align}\label{deg8:eqn1}
	&\frac{4f(\frac{2\pi}{5}|\tau)}{\T_1^2(\frac{\pi}{5}|\tau)}-
	\frac{4f(\frac{\pi}{5}|\tau)}{\T_1^2(\frac{2\pi}{5}|\tau)}\\
	&=\frac{\sqrt{5}\eta^2(5\tau)}{\eta^4(\tau)}
	\left\{\frac{-f(0|\tau)\eta^3(\tau)}{5\eta^3(5\tau)}+\frac{f(\frac{\pi}{2}|\tau)\VT_2(\tau)}{\VT_2(5\tau)}\right.\nonumber\\
	&\qquad \qquad \qquad \qquad  \left. 
	-\frac{qf(\frac{\pi+\pi\tau}{2}|\tau)\VT_3(\tau)}{\VT_3(5\tau)}+\frac{qf(\frac{\pi\tau}{2}|\tau)\VT_4(\tau)}{\VT_4(5\tau)}\right\}.\nonumber
\end{align}	
\end{thm}	
\begin{proof} Using the multiplication formulas for theta functions in \reff{jabel:eqn16}, we conclude that
	\begin{equation}\label{deg8:eqn2}
		\T_1(\frac{\pi}{5}|\tau)\T_1(\frac{2\pi}{5}|\tau)=\sqrt{5}\eta(\tau)\eta(5\tau),
	\end{equation}
	and for $j=2, 3, 4,$
	\begin{equation}\label{deg8:eqn3}
		\T_j(\frac{\pi}{5}|\tau)\T_j(\frac{2\pi}{5}|\tau)=\sqrt{\frac{\eta^5(\tau)\VT_j(5\tau)}{\eta(5\tau)\VT_j(\tau)}}.
	\end{equation}	

	Setting $x=2\pi/5$ and $y=\pi/5$ in Theorem~\ref{liuaddthm} and then using \reff{deg8:eqn2} and \reff{deg8:eqn3} we complete the proof of Theorem~\ref{liuapp:eqn1}.
\end{proof}
\begin{thm}\label{liuapp:eqn2}
	Suppose that $f(z|\tau)$ is an even entire function of $z$ which satisfies the
	functional equations $f(z|\tau)=f(z+\pi|\tau)=q^4 e^{16iz}f(z+\pi\tau|\tau)$.
	Then we  have
	\begin{align}\label{deg8:eqn4}
		&\frac{4q^3f(2\pi\tau|5\tau)}{\T_1^2(\pi\tau|5\tau)}-
		\frac{4f(\pi\tau|5\tau)}{\T_1^2(2\pi\tau|5\tau)}\\
		&=-\frac{\eta^2(\tau)}{\eta^4(5\tau)}
		\left\{\frac{-f(0|5\tau)\eta^3(5\tau)}{\eta^3(\tau)}+\frac{f(\frac{\pi}{2}|5\tau)\VT_2(5\tau)}{\VT_2(\tau)}\right.\nonumber \\
		&\qquad \qquad \qquad \qquad  \left. 
		-\frac{q^5f(\frac{\pi+5\pi\tau}{2}|5\tau)\VT_3(5\tau)}{\VT_3(\tau)}+\frac{q^5f(\frac{5\pi\tau}{2}|5\tau)\VT_4(5\tau)}{\VT_4(\tau)}\right\}.\nonumber
	\end{align}
\end{thm}
\begin{proof} Replacing $\tau$ by $5\tau$ in Theorem~\ref{liuaddthm} and then putting $x=2\pi\tau$ and $y=\pi\tau$ in the resulting equation, we conclude that
	\begin{align}\label{deg8:eqn5}
		&\frac{4q^3f(2\pi\tau|5\tau)}{\T_1^2(\pi\tau|5\tau)}-
		\frac{4f(\pi\tau|5\tau)}{\T_1^2(2\pi\tau|5\tau)}\\
		&=q^{-1/2}{\T_1(\pi\tau|5\tau)\T_1(2\pi\tau|5\tau)}
		\left\{\frac{-f(0|5\tau)}{\T_1^2(\pi\tau|5\tau)\T_1^2(2\pi\tau|5\tau)}+\frac{f(\frac{\pi}{2}|5\tau)}{\T_2^2(\pi\tau|5\tau)\T_2^2(2\pi\tau|5\tau)}\right.\nonumber\\
		&\qquad \qquad \qquad \qquad \qquad \qquad \left. 
		-\frac{q^5f(\frac{\pi+5\pi\tau}{2}|5\tau)}{\T_3^2(\pi\tau|5\tau)\T_3^2(2\pi\tau|5\tau)}+\frac{qf(\frac{5\pi\tau}{2}|\tau)}{\T_4^2(\pi\tau|5\tau)\T_4^2(2\pi\tau|5\tau)}\right\}.\nonumber
	\end{align}

	Using the multiplication formulas for theta functions in \reff{jabel:eqn17} we can deduce that
	\begin{equation}\label{deg8:eqn6}
		\T_1(\pi\tau|5\tau)\T_1(2\pi\tau|5\tau)=-q^{-1/2}\eta(\tau)\eta(5\tau),
	\end{equation}
	and for $j=2, 3, 4$,
	\begin{equation}\label{deg8:eqn7}
		\T_j^2(\pi\tau|5\tau)\T_j^2(2\pi\tau|5\tau)=q^{-1}\frac{\eta^5(5\tau)\VT_j(\tau)}{\eta(\tau)\VT_j(5\tau)}.
	\end{equation}

	Substituting \reff{deg8:eqn6} and \reff{deg8:eqn7} into \reff{deg8:eqn5} we complete the proof of Theorem~\ref{liuapp:eqn2}.	
\end{proof}
By taking $f(z|\tau)=\T_1^2(5z|\tau)/{\T_1^2(z|\tau)}$ in Theorem~\ref{liuapp:eqn1} and simplifying we obtain that  \cite[p.~1522]{Shen95}
\begin{equation}\label{deg8:eqn8}
	\frac{\VT_2(5\tau)}{\VT_2(\tau)}-\frac{\VT_3(5\tau)}{\VT_3(\tau)}
	+\frac{\VT_4(5\tau)}{\VT_4(\tau)}
	=\frac{5\eta^3(5\tau)}{\eta^3(\tau)}.
\end{equation}

Letting $f(z|\tau)=\T_1^2(z|\tau/5)/{\T_1^2(z|\tau)}$ in Theorem~\ref{liuapp:eqn2} and simplifying we arrive at
Ramanujan's identity \cite[p.~276, Eq.(12.32)]{Berndt1991}
\begin{equation}\label{deg8:eqn9}
	\frac{\VT_2(\tau)}{\VT_2(5\tau)}-\frac{\VT_3(\tau)}{\VT_3(5\tau)}
	+\frac{\VT_4(\tau)}{\VT_4(5\tau)}=\frac{\eta^3(\tau)}{\eta^3(5\tau)}.
\end{equation}

By taking $f(z|\tau)=\T_1^8(z|\tau)$ in Theorem~\ref{liuapp:eqn1} and then using Proposition~\ref{6RRCthm:n6}, we can find that
\begin{equation}\label{deg8:eqn10}
	\frac{\VT_2^9(\tau)}{\VT_2(5\tau)}-\frac{\VT_3^9(\tau)}{\VT_3(5\tau)}	
	+\frac{\VT_4^9(\tau)}{\VT_4(5\tau)}=
	2500\eta(\tau)\eta^7(5\tau)+220\eta^7(\tau)\eta(5\tau),
\end{equation}

Applying the imaginary transformations to  both sides of the above equation, we conclude that
\begin{equation}\label{deg8:eqn11}
	\frac{\VT_2^9(5\tau)}{\VT_2(\tau)}-\frac{\VT_3^9(5\tau)}{\VT_3(\tau)}	
	+\frac{\VT_4^9(5\tau)}{\VT_4(\tau)}=
	44\eta(\tau)\eta^7(5\tau)+4\eta^7(\tau)\eta(5\tau).
	\end{equation}
The above two identities can be found in \cite[Theorem~4]{Liu2004RockyMountain}.

 By taking  $f(z|\tau)=\T_1(3z|\tau)/{\T_1(z|\tau)}$ in Theorem~\ref{liuaddthm} and simplifying we can deduce that
	\begin{align}\label{deg8:eqn12}
	&\frac{\VT_2(\tau)} {\VT_2(3\tau)\T_2^2(x|\tau)}-\frac{\VT_3(\tau)} {\VT_3(3\tau)\T_3^2(x|\tau)}+\frac{\VT_4(\tau)} {\VT_4(3\tau)\T_4^2(x|\tau)}\\
	&=\frac{4\T_1(3x|\tau)}{\T_1^2(2x|\tau)\T_1(3x|3\tau)}-\frac{\eta^3(\tau)}{\eta^3(3\tau)\T_1^2(x|\tau)}.\nonumber
\end{align}

Letting $x\to 0$ in  both sides of the above equation and making some elementary calculations, we deduce that
\begin{align}\label{deg8:eqn13}
	&\frac{1}{2}\(3L(3\tau)-L(\tau)\)\\
	&=4\eta^3(\tau)\eta^3(3\tau)
	\(\frac{1}{\VT_2(\tau)\VT_2(3\tau)}-\frac{1}{\VT_3(\tau)\VT_3(3\tau)}+\frac{1}{\VT_4(\tau)\VT_4(3\tau)}\).	\nonumber
\end{align}

Dividing both sides of \reff{jabel:eqn23} by $y-x$ and then letting $y\to x$, we get the following theorem.
\begin{thm}\label{liuthm} If $f(z|\tau)$ is an even entire function of $z$ which satisfies the functional equations 
	\begin{equation}\label{newliu:eqn1}
		f(z|\tau)=f(z+\pi|\tau)=q^{4}e^{16iz} f(z+\pi\tau|\tau),
	\end{equation}	
	then we have
	\begin{align}\label{newliu:eqn2}
		&\(\log f\)'(x|\tau)-4\(\log \T_1\)'(2x|\tau)\\
		&=\frac{\VT_1'(\tau)\T_1^3(2x|\tau)}{4f(x|\tau)}
		\(-\frac{f(0|\tau)}{\T_1^4(x|\tau)}+\frac{f(\frac{\pi}{2}|\tau)}{\T_2^4(x|\tau)}-\frac{qf(\frac{\pi+\pi\tau}{2}|\tau)}{\T_3^4(x|\tau)}+\frac{qf(\frac{\pi\tau}{2}|\tau)}{\T_4^4(x|\tau)}\).\nonumber
	\end{align}	
\end{thm}
\begin{prop}\label{newapp:n3}Let $L(\tau)$ be the Eisenstein series $E_2(\tau)$ defined by \reff{jabel:eqn9}. Then we have
	\begin{equation}\label{newliu:eqn3}
		\frac{1}{6} \(7L(7\tau)-L(\tau)\)=\frac{\VT_1'(\tau)^3}{49\VT_1'(7\tau)}
		\(\frac{\VT_2(7\tau)}{\VT_2^3(\tau)}-\frac{\VT_3(7\tau)}{\VT_3^3(\tau)}
		+\frac{\VT_4(7\tau)}{\VT_4^3(\tau)}\),
	\end{equation}	
	and 
	\begin{equation}\label{newliu:eqn4}
		\frac{1}{6} \(7L(7\tau)-L(\tau)\)=\frac{\VT_1'(7\tau)^3}{\VT_1'(\tau)}
		\(\frac{\VT_2(\tau)}{\VT_2^3(7\tau)}-\frac{\VT_3(\tau)}{\VT_3^3(7\tau)}
		+\frac{\VT_4(\tau)}{\VT_4^3(7\tau)}\).
	\end{equation}	
\end{prop}
The identity \reff{newliu:eqn3} can be found in  \cite[Proposition~4.3]{Liu2010Mysore}.
\begin{proof} If we take $f(z|\tau)=\T_1(z|\tau)\T_1(7z|7\tau)$ in Theorem~\ref{liuthm}, then we have
	\begin{align*}
		&4(\log \T_1)'(2x|\tau)-(\log \T_1)'(x|\tau)-7 (\log \T_1)'(7x|7\tau)\\
		&=\frac{\VT_1'(\tau)\T_1^3(2x|\tau)}{4\T_1(x|\tau)\T_1(7x|7\tau)}
		\(\frac{\VT_2(\tau)\VT_2(7\tau)}{\T_2^4(x|\tau)}-\frac{\VT_3(\tau)\VT_3(7\tau)}{\T_3^4(x|\tau)}+\frac{\VT_4(\tau)\VT_4(7\tau)}{\T_4^4(x|\tau)}\).\nonumber
	\end{align*}	

	Applying the asymptotic formula for $(\log \T_1)'(x|\tau)$ in \reff{jabel:eqn14} to the left-hand side of the above equation  we deduce that near $x=0$, 
	\begin{align*}
		&\frac{1}{6} \(7L(7\tau)-L(\tau)\)x+O(x^3)\\
		&=\frac{\VT_1'(\tau)\T_1^3(2x|\tau)}{56\T_1(x|\tau)\T_1(7x|7\tau)}\(\frac{\VT_2(\tau)\VT_2(7\tau)}{\T_2^4(x|\tau)}-\frac{\VT_3(\tau)\VT_3(7\tau)}{\T_3^4(x|\tau)}+\frac{\VT_4(\tau)\VT_4(7\tau)}{\T_4^4(x|\tau)}\).
	\end{align*}
	Dividing both sides of the above equation by $x$ and then letting $x\to 0$ yields \reff{newliu:eqn3}.
	
	 Applying the imaginary transformations in \reff{jabel:eqn20} and the modular transformation formula for $L(\tau)$ in \reff{jabel:eqn22} to \reff{newliu:eqn3} we can arrive at \reff{newliu:eqn4}.
	\end{proof}

\begin{prop}\label{newapp:n4}Let $L(\tau)=E_2(\tau)$ be the Eisenstein series defined by \reff{jabel:eqn9}. Then we have
	\begin{align}\label{newliu:eqn5}
		&25L(5\tau)+9L(3\tau)-8L(\tau)\\
		&=
		\frac{2\VT_1'(\tau)^4}{5\VT_1'(3\tau)\VT_1'(5\tau)}
		\(\frac{\VT_2(3\tau)\VT_2(5\tau)}{\VT_2^4(\tau)}
		-\frac{\VT_3(3\tau)\VT_3(5\tau)}{\VT_3^4(\tau)}+\frac{\VT_4(3\tau)\VT_4(5\tau)}{\VT_4^4(\tau)}\), \nonumber
	\end{align}	
	and
	\begin{align}\label{newliu:eqn6}
		&\frac{1}{6}\(8L(15\tau)-L(3\tau)-L(5\tau)\)\\
		&=\frac{\VT_1'(15\tau)^4}{\VT_1'(\tau)\VT_1'(5\tau)}
		\(\frac{\VT_2(3\tau)\VT_2(5\tau)}{\VT_2^4(15\tau)}
		-\frac{\VT_3(3\tau)\VT_3(5\tau)}{\VT_3^4(15\tau)}
		+\frac{\VT_4(3\tau)\VT_4(5\tau)}{\VT_4^4(15\tau)}\). \nonumber
	\end{align}
\end{prop}
\begin{proof} By taking $f(z|\tau)=\T_1(3z|3\tau)\T_1(5z|5\tau)$ in Theorem~\ref{liuthm},  we deduce that
	\begin{align*}
		&4(\log \T_1)'(2x|\tau)-3(\log \T_1)'(3x|3\tau)-5(\log \T_1)'(5x|5\tau)\\
		&=\frac{\VT_1'(\tau)\T_1^3(2x|\tau)}{4\T_1(3x|3\tau)\T_1(5x|5\tau)}
		\(\frac{\VT_2(3\tau)\VT_2(5\tau)}{\T_2^4(x|\tau)}-\frac{\VT_3(3\tau)\VT_3(5\tau)}{\T_3^4(x|\tau)}+\frac{\VT_4(3\tau)\VT_4(5\tau)}{\T_4^4(x|\tau)}\).
	\end{align*}
	It follows that near $x=0,$	
	\begin{align*}
		&\frac{1}{3}\(25L(5\tau)+9L(3\tau)-8L(\tau)\)x+O(x^3)\\
		&=\frac{\VT_1'(\tau)\T_1^3(2x|\tau)}{4\T_1(3x|3\tau)\T_1(5x|5\tau)}
		\(\frac{\VT_2(3\tau)\VT_2(5\tau)}{\T_2^4(x|\tau)}-\frac{\VT_3(3\tau)\VT_3(5\tau)}{\T_3^4(x|\tau)}+\frac{\VT_4(3\tau)\VT_4(5\tau)}{\T_4^4(x|\tau)}\).
	\end{align*}
	Dividing both sides of the above equation by $x$ and then letting $x\to 0$ yields \reff{newliu:eqn5}.
	
	Setting $f(z|\tau)=\T_1(z|\frac{\tau}{3})\T_1(z|\frac{\tau}{5})$ in Theorem~\ref{liuthm} and then replacing $\tau$ by $15\tau$, we deduce that
	\begin{align*}
		&(\log \T_1)'(x|3\tau)+(\log \T_1)'(x|5\tau)-4(\log \T_1)'(2x|15\tau)\\
		&=\frac{\VT_1'(15\tau)\T_1^3(2x|15\tau)}{4\T_1(x|3\tau)\T_1(x|5\tau)}
		\(\frac{\VT_2(\tau)\VT_2(5\tau)}{\T_2^4(z|15\tau)}-\frac{\VT_3(\tau)\VT_3(5\tau)}{\T_3^4(z|15\tau)}+\frac{\VT_4(\tau)\VT_4(5\tau)}{\T_4^4(z|15\tau)}\).
	\end{align*}
	It follows that near $x=0,$
	\begin{align*}
		&\frac{1}{3}\(8L(15\tau)-L(3\tau)-L(5\tau)\)x+O(x^3)\\
		&=\frac{\VT_1'(15\tau)\T_1^3(2x|15\tau)}{4\T_1(x|3\tau)\T_1(x|5\tau)}
		\(\frac{\VT_2(\tau)\VT_2(5\tau)}{\T_2^4(z|15\tau)}-\frac{\VT_3(\tau)\VT_3(5\tau)}{\T_3^4(z|15\tau)}+\frac{\VT_4(\tau)\VT_4(5\tau)}{\T_4^4(z|15\tau)}\).
	\end{align*}
	Dividing both sides of the above equation by $x$ and then letting $x\to 0$ yields \reff{newliu:eqn6}.
\end{proof}

By taking $f(z|\tau)=\T_1^2(4z|4\tau)$ in Theorem~\ref{liuthm} we can easily  arrive at the Jacobi four-square identity
\[
\(\sum_{n=-\infty}^\infty q^{n^2} \)^4
=1+8\sum_{n=1}^\infty \frac{nq^n}{1-q^n}-32\sum_{n=1}^\infty \frac{nq^{4n}}{1-q^{4n}}.
\]

\begin{prop}\label{newapp:n5}Let $L(\tau)=E_2(\tau)$ be the Eisenstein series defined by \reff{jabel:eqn9}. Then we have
	\begin{align}\label{newliu:eqn7}
		\frac{1}{2}\(3L(3\tau)-L(\tau)\)=
		\frac{\VT_1'(\tau)^5}{81\VT_1'(3\tau)^3}
		\(\frac{\VT_2^3(3\tau)}{\VT_2^5(\tau)}
		-\frac{\VT_3^3(3\tau)}{\VT_3^5(\tau)}+\frac{\VT_4^3(3\tau)}{\VT_4^5(\tau)}\),
	\end{align}	
	and
	\begin{align}\label{newliu:eqn8}
		\frac{1}{2}\(3L(3\tau)-L(\tau)\)
		=\frac{\VT_1'(3\tau)^5}{\VT_1'(\tau)^3}
		\(\frac{\VT_2^3(\tau)}{\VT_2^5(3\tau)}
		-\frac{\VT_3^3(\tau)}{\VT_3^5(3\tau)}
		+\frac{\VT_4^3(\tau)}{\VT_4^5(3\tau)}\). 
	\end{align}
\end{prop}
The identity in \reff{newliu:eqn8} can be found in  \cite[Proposition~4.2]{Liu2010Mysore}.
\begin{proof} If we specialize Theorem~\ref{liuthm} to the case when $f(z|\tau)=\T_1^3(3z|3\tau)/{\T_1(z|\tau)}$, we obtain that
	\begin{align*}
		&4(\log \T_1)'(2x|\tau)+(\log \T_1)'(x|\tau)-9\(\log \T_1\)'(3x|3\tau)\\
		&=\frac{\VT_1'(\tau)\T_1(x|\tau)\T_1^3(2x|\tau)}{4\T_1^3(3x|3\tau)}
		\(\frac{\VT_2^3(3\tau)}{\VT_2(\tau)\T_2^4(x|\tau)}-\frac{\VT_3^3(3\tau)}{\VT_3(\tau)\T_3^4(x|\tau)}+\frac{\VT_4^3(3\tau)}{\VT_4(\tau)\T_4^4(x|\tau)}\).
	\end{align*}
	From this equation we can find that near $x=0,$	
	\begin{align*}
		&\(9L(3\tau)-3L(3\tau)\)x+O(x^3)\\
		&=\frac{\VT_1'(\tau)\T_1(x|\tau)\T_1^3(2x|\tau)}{4\T_1^3(3x|3\tau)}
		\(\frac{\VT_2^3(3\tau)}{\VT_2(\tau)\T_2^4(x|\tau)}-\frac{\VT_3^3(3\tau)}{\VT_3(\tau)\T_3^4(x|\tau)}+\frac{\VT_4^3(3\tau)}{\VT_4(\tau)\T_4^4(x|\tau)}\).
	\end{align*}
	Dividing both sides of the above equation by $x$ and then letting $x\to 0$ yields \reff{newliu:eqn7}.
	
	By taking $f(z|\tau)=\T_1^3(z|\frac{\tau}{3})/{\T_1(z|\tau)}$ in Theorem~\ref{liuthm} and making some calculations, we can get \reff{newliu:eqn8}.
\end{proof}

Obviously, we have not exhausted the applications of Theorem~\ref{liuaddthm}, but I think this paper has shown the importance of it. Other applications of this theorem, especially  to Appell--Lerch functions, need to be explored.

%%%%%%%%%%%%%%%%%%%%%%%%%%%%%%%%%%%%%%%%%%%%%%%%%%%%%%%%%%%%%%%%%%%%%%%%

\paragraph {\bf Acknowledgements}
I sincerely thank  Bruce Berndt for his consistent  encouragement and support for my work in theta functions over the past 20 years. I am grateful to Bruce Berndt and the referee for careful reading of the original manuscript  of  this  paper,  proposing  some  corrections  and  many  constructive  and helpful comments that resulted in substantial improvements to the paper. I also thank Dandan Chen for pointing out several misprints of an earlier version of this paper.

\end{document}